\documentclass[12pt]{amsart}
\usepackage{setspace}
\usepackage{hyperref}
\usepackage[alphabetic]{amsrefs}
\usepackage{mathtools}
\usepackage{enumerate}
\usepackage{amsmath, amsthm, amssymb, thmtools, etoolbox}
\usepackage{pgfplots}
\usepackage{amsaddr}
\usepackage{comment}
\usepackage{subcaption}
\usepackage{yfonts}

\usepackage{mathrsfs}
\usetikzlibrary{arrows}
\usepackage{tikz-cd}
\usetikzlibrary{calc}
\allowdisplaybreaks

\theoremstyle{plain}

\newtheorem{thm}{Theorem}[section]
\newtheorem{defn}[thm]{Definition}
\newtheorem{rem}[thm]{Remark}
\newtheorem{prop}[thm]{Proposition}

\newtheorem{lem}[thm]{Lemma}
\newtheorem{nota}[thm]{Notation}

\numberwithin{equation}{section}

\begin{document}
\begin{abstract}
We introduce a new approach to confusability in a quantum channel, namely  quantum confusability multigraph, which incorporates the output information into the graphical structure. By``counting" the edges between two  vertices of this confusability multigraph, one recovers the traditional confusability ``single-edged" graph of the channel. With this physical motivation, we therefore develop a theory of quantum multigraphs from Weaver's quantum relations point of view and explore its quantum graph theoretic properties.
Finally, we provide a necessary and sufficient condition characterizing those quantum multigraphs that arise as quantum confusability multigraphs.
\end{abstract}
\title{A multigraph approach to confusability in quantum channels}

\author{Sk Asfaq Hossain}
\address{Department of Mathematics\\Indian Institute of Science, Education and Ressearch, Bhopal\\Bhopal Bypass road, Bhopal - 462066, India\\\textnormal{email: \texttt{asfaq1994@gmail.com}}
}
\author{Angshuman Bhattacharya} 
\address{Department of Mathematics\\Indian Institute of Science, Education and Ressearch, Bhopal\\Bhopal Bypass road, Bhopal - 462066, India\\\textnormal{email: \texttt{angshu@iiserb.ac.in}}
}

\maketitle
%\tableofcontents
\section{\textbf{Introduction}}
Quantum graphs are noncommutative generalizations of classical graphs, which are motivated from both physical and mathematical perspectives. From a physical viewpoint, a quantum graph appears as the confusability graph of a quantum channel, which serves as a fundamental invariant in the study of zero-error communication over noisy channels (\cite{6319408}, \cite{shannon2003zero}). 

From a mathematical viewpoint, a finite graph (classical) can be interpreted as a relation on a finite set. Accordingly, a quantum graph may be regarded as a quantum relation between quantum sets, that is, finite-dimensional algebras (\cite{MR2908249}). Historically, quantum graphs are introduced as operator systems arising as confusability graphs of quantum channels, and hence form a special subclass of quantum relations on the input algebra. In this work, however, we adopt a broader perspective and treat general quantum relations as quantum graphs, motivated by the classical analogy. We will therefore use the terms ``relation'' and ``graph'' interchangeably. Under this viewpoint, the confusability graph of a quantum channel corresponds to a particular class of quantum graphs which are operator systems. This convention is justified by recent developments indicating that many graph-theoretic notions and objects related to graphs are described for general quantum relations on finite quantum sets (\cite{MR4706978}, \cite{brannan2023quantum}, \cite{Brannan2022}).

The primary objective of this work is to introduce and study a notion of quantum multigraph, motivated by both physical and mathematical considerations. Classically, a finite multigraph is defined as a quadruple $(V,E,s,t)$, where $V$ and $E$ are finite sets of vertices and edges, respectively, and $s,t:E \to V$ denote the source and target maps. By labeling or counting multiple edges between a pair of vertices, any edge $\tau$ can be represented as
\[
\tau = (i,j)r, \quad \text{where } i,j \in V, \; r \in \{1,\dots,n\},
\]
with $s(\tau)=i$ and $t(\tau)=j$ (see definition 3.24 of \cite{GOSWAMI2026673}). We refer to this representation as a labeled multigraph. All multigraphs considered in this article will be labeled multigraphs.

To physically motivate the definition of a quantum multigraph, we introduce the notion of a confusability multigraph associated with a quantum channel. The standard confusability graph captures which inputs can be confused, but does not quantify the extent of this confusability. This distinction is easily visible in the classical setting. For a classical channel, the confusability graph has the input alphabet as its vertex set, and two inputs $x_1$ and $x_2$ are connected by an edge if there exists an output $y$ such that $p(y|x_1)p(y|x_2) \neq 0$. However, this graph does not record how many such outputs contribute to the confusion.

To address this limitation, we define the confusability multigraph as follows: for each output $y$ satisfying $p(y|x_1)p(y|x_2) \neq 0$, we introduce an edge between $x_1$ and $x_2$ labeled by $y$. This construction produces a labeled multigraph, which we call the confusability multigraph. The usual confusability graph can be recovered by replacing all the multiple edges between a pair of vertices with a single edge. We generalize these ideas to the quantum setting, where the extent of confusability is captured by elements of the output algebra.

Incorporating output information into confusability structures has been considered previously from information theoretic viewpoint, for instance through confusability hypergraphs of classical channels, where nonsignaling assisted zero-error capacities are expressed in terms of fractional packing numbers of the confusability hypergraph (\cite{cubitt2011zero}). In the quantum setting, analogous quantities are described via similar semidefinite programs involving the Kraus operators of the channel (\cite{duan2015no}). This suggests that the Kraus space, that is, linear span of the Kraus operators, provides a finer description of confusability than the traditional confusability graph itself.

Our construction of the quantum confusability multigraph  is done by exploiting the Stinespring representation of a completely positive map, replacing the usual scalar-valued inner product with an operator-valued inner product on the bigger Hilbert space. For a quantum channel $\Phi: B(H_{in}) \to B(H_{out})$, our computation shows (remark \ref{decomp_rem}) that the associated confusability multigraph $\tilde{S}_\Phi$ is, as a vector space, isomorphic to
\[
\mathcal{K}^* \otimes \mathcal{K}, \quad \text{and ``counting quantum edges'' we get } \quad \mathcal{K}^*\mathcal{K} = S_\Phi,
\]
where $\mathcal{K}$ denotes the Kraus space of $\Phi$ and $S_\Phi$ is the usual (single-edged) quantum confusability graph. This shows that the multigraph structure retains the full Kraus space information, in contrast to its single-edged counterpart. Therefore our quantum confusability multigraph provides a richer combinatorial framework for analyzing different zero error communications in quantum channels. It should be noted that above isomorphism is just a vector space isomorphism, the confusability multigraph itself is intrinsically a different object lying in $B(H_{in})\otimes B(H_{out})^{op}$ where $B(H_{out})^{op}$ is the opposite algebra of $B(H_{out})$ obtained by reversing the multiplication in $B(H_{out})$.  

From a mathematical perspective, a labeled classical multigraph can be viewed as a subset of $X \times X \times Y$, where $X$ is the vertex set and $Y$ is the label set used to label the edges. Following this idea, we introduce the notion of a multi-relation on a pair of finite sets $(X,Y)$. Extending Weaver's framework of quantum relations (\cite{MR2908249}), we define quantum multi-relations on pairs of finite-dimensional algebras $(M,N)$, and develop associated notions such as edge indicators and different adjacency operators. In particular, the quantum multigraph described on a pair of algebras $(M,N)$, where $M$ is the quantum vertex space and $N$ is the quantum label set, produces a positive operator in $M\otimes M^{op}$ (proposition \ref{P_v-S_v_relation}) which serves as ``weighted edge indicator" associated with the quantum multigraph. For classical multigraphs, this positive operator records the number of edges between two fixed vertices. This construction  generalizes definition 2.13 of \cite{gromada2022some}. One important distinction between the theory of quantum multigraphs and the existing theory of quantum ``single-edged"  graphs is the absence of the three equivalent formulations available in the latter (\cite{chirvasitu2022random}, \cite{MR4706978}). In this context, the most natural approach would be to adopt the perspective of Weaver’s theory of quantum relations  (\cite{MR2908249}), which we have followed in this article.

We further investigate a distinguished class of quantum multi-relations, namely decomposable multi-relations, and show that every symmetric, decomposable quantum multi-relation arises as a quantum confusability multigraph. This provides a multigraph analogue of proposition 3.12 of \cite{MR4707042} or lemma 2 in \cite{duan2009super}.

Finally, we outline the structure of the paper. Since our construction relies on operator-valued inner products, we begin in Section~\ref{prelim} with a very brief review of Hilbert $C^*$-modules. In Section \ref{conf_mult_channel}, we construct the confusability multigraph of a quantum channel and express it in terms of Kraus operators. In Section \ref{quant_mult}, we introduce classical and quantum multi-relations. Although these notions extend to general von Neumann algebras, we restrict attention to finite-dimensional algebras to emphasize graph-theoretic aspects. In Section \ref{decomp_mult}, we explore the notion of a ``decomposable" multi-relation, which can also be seen as a composition of two bipartite quantum relations. This class of multi-relations serves as a  physically meaningful example within our multi-relational framework. Towards the end, we provide necessary and sufficient conditions on these type of multi-relations for them to arise as confusability multigraphs of some completely postive maps.
\section{\textbf{Preliminaries on Hilbert C* Modules}}\label{prelim}
Let $\mathcal{A}$ be a unital C* Algebra. 
\begin{defn}
A pre-Hilbert $\mathcal{A}$ module $\mathcal{M}$ over a C* Algebra $\mathcal{A}$ is a right $A$ module $\mathcal{M}$ equipped with an $A$-valued inner product $\langle .,. \rangle_{A}: \mathcal{M}\times \mathcal{M}\rightarrow \mathcal{A}$ such that following conditions hold:
\begin{enumerate}
    \item $\langle x,x\rangle_{\mathcal{A}} \geq 0$ for any $x\in \mathcal{M}$.
    \item $\langle x,x\rangle_{\mathcal{A}}=0$ if and only if $x=0$.
    \item $\langle x,y \rangle_{\mathcal{A}}=\langle y,x\rangle_{\mathcal{A}} ^*$ for all $x,y\in\mathcal{M}$.
A pre-Hilbert $A$ module is a Hilbert $A$ module if it is complete with respect to the complex valued norm on $\mathcal{M}$ which is given by
\begin{equation*}
    ||x||:=||\langle x,x \rangle_{\mathcal{A}} ||^{1/2}
\end{equation*}
\end{enumerate}
\end{defn}
\subsubsection*{Examples:}
\begin{enumerate}
\item Given any Hilbert space $H$ and $\mathcal{A}$, the vector space $H\otimes\mathcal{A}$ is an right $\mathcal{A}$ module with $A$ valued inner product given by,
\begin{equation*}
    \langle \xi\otimes a, \eta \otimes b\rangle _{\mathcal{A}}= \langle \xi,\eta\rangle a^*b.
\end{equation*}
\item Any right ideal $I$ of $A$ can be seen as right $A$ module with $A$ valued inner product 
\begin{equation*}
\langle a,b\rangle_{\mathcal{A}} :=a^*b
\end{equation*}
\end{enumerate}
\begin{defn}
An $\mathcal{A}$ linear operator $T:\mathcal{M}\rightarrow \mathcal{N}$ between two Hilbert C* modules $\mathcal{M}$ and $\mathcal{N}$ over $A$ is called \textbf{adjointable} if there exists another $\mathcal{A}$ linear map $T^*:\mathcal{N}\rightarrow \mathcal{M}$ such that
\begin{equation*}
    \langle x,Ty \rangle_{\mathcal{A}}=\langle T^*x,y \rangle_{\mathcal{A}}\quad\text{for all}\quad x\in \mathcal{M}, y\in \mathcal{N}.
\end{equation*}
The algebra of all adjointable operators from $\mathcal{M}$ to $\mathcal{N}$ is denoted by $\mathcal{L}(\mathcal{M}, \mathcal{N})$. If $\mathcal{N}=\mathcal{M}$ then we write $\mathcal{L}(\mathcal{M}, \mathcal{N})=\mathcal{L}(\mathcal{M})$.
\end{defn}
\textbf{Example:}
If $H$ is a Hilbert space, then $\mathcal{L}(H\otimes \mathcal{A})$ is isomorphic to $B(H)\otimes \mathcal{A}$ as C* algebras. \\
\textbf{Convention: Unless otherwise stated, all Hilbert spaces are finite-dimensional, and inner products are linear in the second variable.}
\section{\textbf{Confusability Multigraph of a Quantum Channel}}\label{conf_mult_channel}
Let $\{H^a_{in}\:|\:a\}$ and $\{H^b_{out}\:|\:b\}$ be two finite collections of finite dimensional Hilbert space. We define two finite dimensional algebras (seen as input and output algebras) by
\begin{align*}
    I=\bigoplus_a B(H^a_{in})\quad\text{and}\quad O=\bigoplus_b B(H_b)
\end{align*}
where $B(H^a_{in})$ and $B(H^b_{in})$ are the algebra of bounded operators on Hilbert spaces $H^a_{in}$ and $H^{in}_b$ respectively. The algebras $I$ and $O$ are naturally subalgebras of $B(H_{in})$ and $B(H_{out})$ where $$H_{in}=\bigoplus_a H^a_{in}\quad\text{and}\quad H_{out}=\bigoplus_b H^b_{out}.$$
We equip both the algebras $I$ and $O$ with traces (non-normalized) coming from $B(H_{in})$ and $B(H_{out})$ respectively making these algebras Hilbert spaces with respect to Hilbert Schmidt inner product. \par
A quantum channel is a ``trace preserving completely positive map" $\Phi:I\rightarrow O$. As $\Phi$ is trace preserving, the adjoint map (taken with respect to the Hilbert Schmidt inner product) $\Phi^*:O\rightarrow I$ turns is a unital completely positive map (UCP in short) from $O$ to $I$.\par
To describe the confusability graph associated to a quantum channel $\Phi$, we follow the approach in section 6.2 of \cite{MR4706978}. A Stinespring representation of the UCP map $\Phi^*:O\rightarrow I\subseteq  B(H_{in})$ is given by the triplet $(\pi, V,K)$, where $K$ is a larger Hilbert space interpreted as ``channel plus environment",  the map $\pi:O\rightarrow B(K)$ is a $C^*$ algebra homomorphism and  $V:H_{in}\rightarrow K$ is an isometry such that the diagram in figure \ref{fig:Stinespring_rep_phi*} commutes.
\begin{figure}[h]
    \centering
   \begin{tikzcd}[row sep=large, column sep=large]
& B(K) \arrow[d, "V^{*}(\,\cdot\,)V"] \\
O \arrow[ur, "\pi"] \arrow[r, "\Phi^{*}"'] 
& B(H_{in})
\end{tikzcd}
    \caption{Stinespring representation for $\Phi^*:O\rightarrow I$}
    \label{fig:Stinespring_rep_phi*}
\end{figure}

It is called ``minimal" iff $K=\overline{span}\:\{\pi(T)V\xi\:|\:T\in O, \xi\in H_{in}\}$. The quantum confusability graph of a quantum channel $\Phi$ is given by theorem 6.4 of \cite{MR4706978}.  
\begin{thm}
Let $\Phi:I\rightarrow O$ be a quantum channel. Let $(\pi,V,K)$ be  a Stinespring representation of the UCP map $\Phi^*:O\rightarrow I$, then the  quantum confusability graph $S_\Phi\in B(H_{in})$ associated with $\Phi$ is given by $$S_\Phi=V^* \pi(O)' V$$ where $\pi(O)'=\{S\in B(K)\:|\: S\pi(T)=\pi(T)S\:\:\forall\:\: T\in O\}$. It further follows that, $S_{\Phi}$ is independent of the chosen Stinespring representation $(\pi,V,K)$.
\end{thm}

To produce a \textbf{quantum confusability multigraph} from a quantum channel $\Phi$, we look at a special kind of Stinespring representation of $\Phi$ where the larger Hilbert space $K$ admits an output algebra valued inner product rather than a complex valued one. This approach aligns with our philosophy as the extent of confusability is captured by elements of the output algebra. Before formalizing all this, we introduce some notations.

As $I$ and $O$ both are finite direct sum of matrix algebras, they can be written as
\begin{equation*}
    I=span\:\{e^a_{ij}\:|\: i,j=1,..,n_a,; a=1..,n\}\quad\text{and}\quad O=span\:\{f^b_{kl}\:|\: ,k,l=1...,m_b; b=1,..,m \}
\end{equation*}
where $e^a_{ij}$'s and $f^b_{kl}$'s are matrix units of $a$ th component of $I$ and  $b$ th component of $O$ respectively. It follows that,
\begin{align*}
    H_{in}=span\{e^a_i\:|\:i,a\}\quad\text{and}\quad H_{out}=span\{f^b_k\:|\:k,b\}.
\end{align*}
\begin{defn}
Given a C* algebra $O$, the \textbf{opposite algebra  $O^{op}$} is obtained from $O$ by reversing the multiplication, that is, as a set and a vector space $O^{op}=O$ but the multiplication ``$*$" in $O^{op}$ is given by $$a*b=ba\quad \text{for all}\quad a,b\in O^{op}.$$
\end{defn}
\begin{rem}\label{opp_action}
     Let us denote $\overline{H_{out}}$ to be the conjugate Hilbert space of $H_{out}$ and $\xi \mapsto \overline{\xi}$ be the antilinear isomoprhism from $H_{out}$ to $\overline{H_{out}}$. The opposite algebra $O^{op}$ can be faithfully represented on the conjugate Hilbert space $\overline{H_{out}}$ via the action $\pi_{op}:O^{op}\rightarrow B(\overline{H_{out}})$ where
\begin{equation*}
    \pi_{op}(T)(\overline{\xi})=\overline{T^*\xi}\quad\forall \quad T\in O^{op}. 
\end{equation*} 
\end{rem}
We will follows these conventions throughout this article.
\begin{defn}
Given a linear map $\Phi:I \rightarrow O$, let us consider the right Hilbert $O^{op}$ module $\mathcal{M}_\Phi=H_{in}\otimes O^{op}$, where $O^{op}$ is the opposite algebra of $O$. We define an operator $C_\Phi\in\mathcal{L}(\mathcal{M}_{\Phi})$ by the following prescription:
\begin{equation*}
    C_{\Phi}(e^a_{j}\otimes 1):=\sum_{i}e^{a}_{i}\otimes \Phi(e^{a}_{ji})
\end{equation*}
where $e^a_i\otimes 1$ is a basis element of the right $O^{op}$ module $H_{in}\otimes O^{op}$  and $\Phi(e^{a}_{ji})$'s are seen as elements of $O^{op}$ instead of $O$.
\end{defn}
We introduce some more notations.
\begin{nota}\label{rank-1}
For $\xi, \eta\in H$ where $H$ is a Hilbert space, we will denote by $\theta_{\xi, \eta}$ to be the rank one operator defined by
$$\theta_{\xi,\eta}(\eta')=\langle \eta, \eta' \rangle \xi.$$
\end{nota}

\begin{prop}\label{CP_module_correspondence}
   A linear map $\Phi:I\rightarrow O$ is completely positive if and only if $C_{\Phi}\in \mathcal{L}(\mathcal{M}_{\Phi})$ is a positive operator on the Hilbert $O^{op}$ module $\mathcal{M}_{\Phi}$, that is,
   \begin{equation*}
       \langle\xi,C_{\Phi}(\xi)\rangle_{O^{op}}\geq 0\quad\text{for all}\quad\xi\in \mathcal{M}_{\Phi}
   \end{equation*}
\end{prop}
\begin{proof}
We adapt the proof of proposition 6.1 in \cite{MR1976867} in our setting. Let us define a linear functional $s_{\Phi}:I\otimes B(\overline{H_{out}})\rightarrow \mathbb{C}$ by
\begin{equation*}
    s_{\Phi}(T)=\langle \Gamma, \tilde{\Phi} (T)\Gamma\rangle\quad \text{for all}\quad T\in I\otimes B(\overline{H_{out}}),
\end{equation*}
  where $\tilde{\Phi}(T)=(\Phi\otimes id)(T)\in B(H_{out}\otimes \overline{H_{out}})$ and $\Gamma\in H_{out}\otimes \overline{H_{out}}$ is given by
  \begin{align*}
  \Gamma=\sum_{k} f_{k}\otimes \overline{f_k}. 
  \end{align*}
\textbf{Claim:} The map $\Phi$ is CP if and only if $s_{\Phi}$ is positive.\\
\textbf{Proof of claim}: It is easy to see that if $\Phi$ is CP, then $s_{\Phi}$ is a positive functional. \\
Conversely, let $s_{\Phi}$ be a positive functional. It is enough to show that $\Phi_m:=\Phi\otimes id : I\otimes M_m\rightarrow O\otimes  M_m$ is a positive map where $M_m$ is the algebra of $m\times m$ matrices.\par 
Using lemma 3.13 of \cite{MR1976867}, it follows that any positive element of $I\otimes M_m$ is finite sum of elements of the form $X=\sum^m_{k,l=1}x^*_kx_l \otimes E_{kl}$ where $\{x_1,x_2,..,x_m\}\subseteq I$ and $\{E_{kl}\:|\:k,l=1,..,m\}$ is the standard basis of $M_m$ consisting of elementary matrices.
For any $\tilde{\xi}\in H_{out}\otimes \mathbb{C}^m$ and a positive element $X=\sum^m_{k,l=1}x^*_kx_l \otimes E_{kl}\in I\otimes M_m$, we observe that,
\begin{align}
\langle \tilde{\xi}, \Phi_m(X)\tilde{\xi}\rangle &= \sum^m_{i,j=1} \langle \xi_i, \Phi(x_i^*x_j)\xi_j \rangle \notag\\ &= \sum_{i,j,k,l} \langle f_{k}, \Phi(x_i^*x_j)f_{l}\rangle \langle \overline{f_{k}},\: \theta_{\overline{\xi_{i}},\overline{\xi_{j}}}\overline{f_{l}}\rangle
=\langle \Gamma, \tilde\Phi (x_i^*x_j \otimes \theta_{\overline{\xi_i},\overline{\xi_j}}) \Gamma\rangle.  \label{Choi_matrix_prop_pf}
\end{align}
We write $\theta_{\overline{\xi_i},\overline{\xi_j}}=A^*_i A_j$ where $A_i=\theta_{\overline{f_1},\overline{\xi_{i}}}$.  
Therefore combining these observations with equation \ref{Choi_matrix_prop_pf} produces 
\begin{align*}
\langle \tilde{\xi}, \Phi_m(X) \tilde{\xi} \rangle &= \langle \Gamma,\tilde{\Phi} ((\sum_{i} x_i\otimes A_i)^* (\sum_{i} x_i\otimes A_i)) \Gamma \rangle\\
&=s_{\Phi}((\sum_{i} x_i\otimes A_i)^* (\sum_{i} x_i\otimes A_i))\geq 0
\end{align*}
as $s_{\Phi}$ is a positive linear functional. \\
As $\mathcal{L}(\mathcal{M}_{\Phi})\cong B(H_{in})\otimes O^{op}$ as algebras, under this isomorphism the map $C_{\Phi}$ has the following form: 
\begin{equation}\label{Choi_mat_formula}
C_{\Phi}=\sum_{i,j,a}e^{a}_{ij}\otimes \Phi(e^{a}_{ji}).
\end{equation}
From proposition 2.1.3 of \cite{MR2125398}, it follows that $\Phi$ is CP, if and only if $C_{\Phi}$ is a positive element of the C* algebra  $\mathcal{L}(\mathcal{M}_{\Phi})$.
Furthermore, by considering $\tilde{C}_{\Phi}=(id\otimes \pi_{op}) (C_{\Phi})\in B(H_{in}\otimes \overline{H_{out}})$ we observe that $$s_{\Phi}(T)=(tr\otimes tr)(\tilde{C}_{\Phi}T)\quad\text{for all}\quad T\in I\otimes B(\overline{H_{out}}).$$
Using the claim we proved, we observe that $\Phi$ is CP iff $s_{\Phi}$ is positive iff the map $C_{\Phi}:\mathcal{M}_{\Phi}\rightarrow \mathcal{M}_{\Phi}$ is a positive map of Hilbert $O^{op}$ modules. 
\end{proof}
\begin{rem}\label{Choi-type_invar}
  The element $C_{\Phi}\in B(H_{in})\otimes O^{op}$ in equation \ref{Choi_mat_formula} is a ``Choi-type" operator associated with a linear map $\Phi$ whose positivity determines complete positivity of $\Phi$. We will refer it simply as the \textbf{Choi-type invariant} of $\Phi$ in this article.
\end{rem}

\begin{defn}
A \textbf{Stinespring module related to a quantum channel} $\Phi: I\rightarrow O$ is a pair $(\mathcal{E}, W)$ where $\mathcal{E}$ is a Hilbert $O^{op}$  module and $W\in \mathcal{L}(\mathcal{M}_{\Phi}, \mathcal{E})$ such that,
\begin{align*}
W^*W=C_{\Phi}.
\end{align*}
Moreover, A Stinespring module $(\mathcal{E},W)$ is said to be \textbf{``minimal"} if $Range (W)$ is dense in $\mathcal{E}$.
\end{defn}
 The next proposition shows that any Stinepring module  generates a Stinespring representation for the UCP map $\Phi^*$. Before stating the result, we prove a lemma which will be crucial to the constructions later on.
\begin{lem}\label{lemma_theta*}
Given a CP map $\Phi:I\rightarrow O$, the adjoint $\Phi^*:O\rightarrow I$ is given by 
\begin{align*}
    \Phi^*(x)=\sum_{i,a}e^{a}_{ij}\:tr(\Phi(e^{a}_{ji})x) \quad\text{for all} \quad x\in O
\end{align*}
\end{lem}

\begin{proof}
We recall that trace in $B(H_{out})$ makes the algebra $O$ a Hilbert space with respect to the inner product:
\begin{align*}
    \langle a,b \rangle_{tr}=tr(a^*b)\quad\text{for all}\quad a,b\in O.
\end{align*}
Using the fact that $\Phi$ is CP and hence $*$ preserving, we have 
\begin{align*}
\Phi^*(x)= \sum_{i,j,a}e^{a}_{ij}\:\langle e^{a}_{ij},\Phi^*(x)\rangle_{tr}
&=\sum_{i,j,a}e^{a}_{ij}\:\langle \Phi( e^{a}_{ij}),x\rangle_{tr}\\
&=\sum_{i,j,a}e^{a}_{ij}\:tr({\Phi( e^{a}_{ij})}^*x)=\sum_{i,j,a}e^{a}_{ij}\:tr({\Phi( e^{a}_{ji})}x)
\end{align*}
\end{proof}
\begin{prop}\label{module_rep_correspondence}
A Stinespring module $(\mathcal{E}, W)$ related to a quantum channel $\Phi:I\rightarrow O$ produces a Stinespring representation for the UCP map $\Phi^*:O\rightarrow I$, given by the triplet $(\pi,V,\mathcal{E})$, where $\mathcal{E}$ is regarded as a Hilbert space with respect to the inner product $tr(\langle.,.\rangle_{O^{op}})$, the maps $\pi:O\rightarrow B(\mathcal{E})$ and $V:H_{in}\rightarrow \mathcal{E}$ are described as follows:
\begin{align*}
    \pi(x)(\xi)=\xi*x\quad \text{and} \quad V(\eta)=W(\eta\otimes 1)
\end{align*}
where $\xi\in \mathcal{E}$, $\eta\in H_{in}$ and $``*"$ denotes the right $O^{op}$ multiplication in $\mathcal{E}$. Moreover, if $(\mathcal{E}, W)$  is minimal, then the associated Stinespring representation is also minimal.
\end{prop}
\begin{proof}
For $\xi,\eta\in\mathcal{E}$, we write $$\langle \xi, \eta \rangle_{tr}= tr(\langle \xi, \eta \rangle_{O^{op}}).$$
To show that the triplet $(\pi,V,\mathcal{E})$ is a Stinespring representation for $\Phi^*$ we need to prove that $$V^*\pi(x)V(\xi)=\Phi^*(x)(\xi)\quad \text{for all}\quad\xi\in H_{in}, x\in O.$$
We  observe that, for arbitrary $\xi,\eta\in H_{in}$ and $x\in O$ 
\begin{align*}
    \langle \eta, V^*\pi(x)V(\xi) \rangle_{tr}=\langle V(\eta), \pi(x)V(\xi)\rangle_{tr}
    &=tr(\langle W(\eta\otimes 1), W(\xi\otimes x)\rangle_{O^{op}})\\
    &=tr\:(\langle \eta\otimes 1, W^*W(\xi\otimes x) \rangle_{O^{op}})\\
    &=tr\:(\langle \eta\otimes 1, C_{\Phi}(\xi\otimes x) \rangle_{O^{op}} )
\end{align*}
Putting $\xi=e_j^{a}$ in the above expression and using lemma \ref{lemma_theta*} we observe that,
\begin{align*}
\langle \eta, V^*\pi(x)V(e^a_j) \rangle_{tr}=\sum^{n_a}_{i=1} \langle \eta, e^a_{i}\rangle tr(\Phi(e^a_{ji})x)= \langle \eta, \sum^{n_a}_{i=1}tr(\Phi(e^a_{ji})x) e^a_{i}\rangle= \langle \eta, \Phi^*(x)e^{a}_{j}\rangle.
\end{align*}
As $e^a_j$ is an arbitrary basis element of $H_{in}$, the result follows. If $(\mathcal{E}, W)$ is a minimal Stinespring module, then minimality of the associated Stinespring representation is straightforward to check. 
\end{proof}
\begin{rem}
    It should be noted that for any $T$ in either $O$ or $O^{op}$ (both are same as sets), $tr(T)=tr(\pi_{op}(T))$.Therefore, there is no ambiguity when we consider the same trace on $O^{op}$ as $O$ which has been done in the previous proposition.
\end{rem}

\begin{prop}\label{mudule_indep_mult}
   Let $(\mathcal{E},W)$ and  $(\mathcal{E}',W')$ be two Stinespring modules related to a quantum channel $\Phi:I\rightarrow O$. Then following condition holds:
   \begin{equation*}
   W^*\mathcal{L}(\mathcal{E})W=W'^{*}\mathcal{L}(\mathcal{E}') W'\subseteq \mathcal{L}(\mathcal{M}_{\Phi})
   \end{equation*}
   \end{prop}
   \begin{proof}
      Without loss of generality, we assume that $(\mathcal{E}, W)$ is minimal. As $W:\mathcal{M}_{\Phi}\rightarrow \mathcal{E}$ is surjective, we define an $O^{op}$ linear map $U:\mathcal{E}\rightarrow \mathcal{E}'$ by the following description:
      \begin{equation*}
          U(W(\Gamma))=W'(\Gamma)\quad \text{for all} \quad \Gamma \in H_{in}\otimes O^{op}
      \end{equation*}
      To check that $U$ is well defined, we observe that, for any $\Gamma\in H_{in}\otimes O^{op}$,
      \begin{align*}
          W(\Gamma)=0\quad&\text{if and only if}\quad \langle \Gamma, W^*W(\Gamma)\rangle_{O^{op}}=0\\ &\text{if and only if}\quad \langle \Gamma, C_{\Phi}(\Gamma)\rangle_{O^{op}}=0\\
          &\text{if and only if}\quad \langle \Gamma, {W'}^*W'(\Gamma)\rangle_{O^{op}}=0 \quad \text{iff}\quad W'(\Gamma)=0.
     \end{align*}
     The map $U$ is an isometry. For any $\Gamma_1, \Gamma_2\in H_{in}\otimes O^{op}$, we observe that,
 \begin{align*}
         \langle W(\Gamma_1)), W(\Gamma_2)\rangle_{O^{op}}=\langle \Gamma_1, C_{\Phi}(\Gamma_2)\rangle_{O^{op}}&= \langle \Gamma_1,{W'}^*W'(\Gamma_2\rangle_{O^{op}}\\
         &= \langle W'(\Gamma_1),W'(\Gamma_2)\rangle_{O^{op}}.
\end{align*}
    We claim that, $U\in \mathcal{L}(\mathcal{E},\mathcal{E}')$ and $U^*U=id_{\mathcal{E}}$. To show this, it is enough to show that $Range (U)$ is a complemented $O^{op}$ submodule of $\mathcal{E}'$ (proposition 3.6 of \cite{MR1325694}). We observe that, $Range(U)=Range (W')$ and $Range(W')$ is a complemented $O^{op}$ submodule of $\mathcal{E}'$ as $W'\in\mathcal{L}(\mathcal{M}_{\Phi}, \mathcal{E}')$ has finite dimensional range and hence closed (see proposition 3.2 of \cite{MR1325694}). \\
    \textbf{Claim:} $U^*\mathcal{L}(\mathcal{E}')U=\mathcal{L}(\mathcal{E})$.\\
    \textbf{Proof of claim:} It is clear that $U^*\mathcal{L}(\mathcal{E}')U\subseteq \mathcal{L}(\mathcal{E})$. Conversely, for any $T\in\mathcal{L}(\mathcal{E})$, using $U^*U=id_{\mathcal{E}}$, we observe that,
    \begin{align*}
    T=U^*(UTU^*)U\in U^*\mathcal{L}(\mathcal{E}')U.
    \end{align*}
    Using the claim and the fact that $W'=UW$, the result follows.
    \end{proof}
    \begin{defn}
    The \textbf{quantum confusability multigraph} associated with a quantum channel $\Phi:I\rightarrow O$ is given by 
    \begin{equation*}
    \tilde{S_{\Phi}}= W^*\mathcal{L}(\mathcal{E})W\subseteq \mathcal{L} (\mathcal{M}_{\Phi})\cong B(H_{in})\otimes O^{op}
    \end{equation*}
    where $(\mathcal{E}, W)$ is a Stinespring module related to the channel $\Phi$.
\end{defn}
\begin{rem}
One can see that the trace preserving property of a quantum channel does not play any role in the construction of quantum confusability multigraph. Therefore, a quantum confusability multigraph can be associated with any CP map via the same construction. Moreover, proposition \ref{mudule_indep_mult} shows that the quantum confusability multigraph is  independent of the choice of Stinespring module related to $\Phi:I\rightarrow O$.
\end{rem}

\begin{prop}
Let $\Phi:I \rightarrow O$ be a quantum channel. The   quantum confusability single-edged graph is obtained by ``counting" the edges of the quantum confusability multigraph, that is, 
\begin{equation*}
    (id\otimes tr)(\tilde{S_{\Phi}})=S_{\Phi}.
\end{equation*}

\end{prop}
\begin{proof}
Let $(\mathcal{E},W)$ be a minimal Stinespring module related to $\Phi$ and $(\pi, \mathcal{E}, V)$ be the associated Stinespring representation for $\Phi^*$ (see proposition \ref{module_rep_correspondence}). We observe that, $$B(\mathcal{E})\supseteq \pi(O)'=\mathcal{L}(\mathcal{E}). $$ Therefore, the quantum confusability multigraph and the quantum confusability single-edged graph are given by,
\begin{equation*}
\tilde{S}_{\Phi}=W^*\mathcal{L}({\mathcal{E}})W\quad\text{and}\quad S_{\Phi}= V^*  \mathcal{L}({\mathcal{E}}) V 
\end{equation*}
Let us define a linear map $i: H_{in}\rightarrow \mathcal{M}_{\Phi}$ by $i(\xi)=\xi\otimes 1$. Equiping  $\mathcal{M}_{\Phi}$ with inner product $tr(\langle.,.\rangle_{O^{op}})$ we consider the adjoint $i^*:\mathcal{M}_{\Phi}\rightarrow H_{in}$. It is clear that the following diagram \ref{fig:a_conf_mult_to_simp}  commutes and therefore diagram \ref{fig:b_conf_mult_to_simp} also commutes.
\begin{figure}[h]
\centering

\begin{subfigure}{0.45\textwidth}
\centering
\begin{tikzcd}[column sep=4em, row sep=3em]
& \mathcal{E} \\
\mathcal{H}_{\mathrm{in}}
  \arrow[ur, "V"]
  \arrow[r, "i"']
& M_\phi
  \arrow[u, "W"']
\end{tikzcd}
\caption{}
\label{fig:a_conf_mult_to_simp}
\end{subfigure}
\hspace{1em}
\begin{subfigure}{0.45\textwidth}
\centering
\begin{tikzcd}[column sep=4em, row sep=3em]
& \mathcal{L}(\mathcal{E})
  \arrow[dl, "V^*(\cdot)V"']
  \arrow[d, "W^*(\cdot)W"] \\
B(\mathcal{H}_{\mathrm{in}})
& \mathcal{L}(M_\phi)
  \arrow[l, "i^*(\cdot)i"']
\end{tikzcd}
\caption{}
\label{fig:b_conf_mult_to_simp}
\end{subfigure}

\caption{single-edged confusability graph from confusaibility multigraph}
\label{conf_mult_to_simp}
\end{figure}

As $W^*\mathcal{L}(\mathcal{E})W=\tilde{S_{\Phi}}$, to prove the desired result, it is enough to show that 

\begin{align}\label{tr_evaluation_map_i}
    i^*(T)i=(id\otimes tr)(T)\quad \text{for all}\quad T\in \mathcal{L}(\mathcal{M}_{\Phi}).
\end{align}
For $\xi,\eta\in H_{in}$ and $b\in O^{op}$ we observe that, 
\begin{align*}
\langle \eta, i^*(\xi\otimes b )\rangle=tr (\langle \eta\otimes 1, \xi\otimes b\rangle_{O^{op}})=\langle \eta, \xi \rangle tr(b).
\end{align*}
Therefore $$i^*(\xi\otimes b)=\xi\: tr(b).$$
As $H_{in}=span\:\{e^i_a\:|\:i,a\}$, we write $B(H_{in})=span\:\{e^{ab}_{ij}\:|\:i,j,a,b\}$. Therefore any $T\in \mathcal{L}({M_{\Phi}})\cong B(H_{in})\otimes O^{op}$ can be written as $T=\sum_{i,j,a,b}e^{ab}_{ij}\otimes T^{ab}_{ij}$ where $T^{ab}_{ij}\in O^{op}$ for all $i,j,a,b$. For any $\xi\in H_{in}$, we observe that,
\begin{align*}
i^*(T)i(\xi)=i^*(T(\xi\otimes 1))&=i^*(\sum_{i,j,a,b}e^{ab}_{ij}(\xi)\otimes T^{ab}_{ij})\\
&=\sum_{i,j,a,b}e^{ab}_{ij}(\xi)tr(T^{ab}_{ij})=\big((id\otimes tr)(T)\big)(\xi).
\end{align*}
As $\xi\in H_{in}$ is arbitrary, equation \ref{tr_evaluation_map_i} follows.
\end{proof}
\subsection*{Kraus operators:} Every CP map between matrix algebras admit Kraus decomposition which are of special interests to quantum information theorists. Therefore it is important that we describe our confusability multigraph in terms of Kraus operators.
Given a quantum channel $\Phi:I\rightarrow O$, we describe the Kraus operators as follows:\\
We recall that, 
\begin{align*}
H_{in}=\bigoplus_a H^a_{in},\quad I=\bigoplus_a B(H^a_{in}), \quad H_{out}=\bigoplus_b H^b_{out},\quad O=\bigoplus_b B(H^b_{out}).
\end{align*}
By compressing both domain and range, we have a family of CP maps $\{\Phi_{ab}\in B(H^a_{in}, H^b_{out})\:|\:a,b\}$ and hence a family of Kraus operators $\{E_{abk}\:|\:k\}$ related to $\Phi_{ab}$. The map $\Phi:I\rightarrow O$ can be written as 
\begin{align*}
    \Phi(x)=\sum_{a,b,k} (j_b E_{abk} {i_a}^*) x ({i_a} {E_{abk}}^* {j_b}^*)
\end{align*}
where $i_a:H^a_{in}\rightarrow H_{in}$ and $j_b:{H^b_{out}}\rightarrow H_{out}$ are canonical inclusion maps. Renaming the indices $a,k$  of the Kraus operators, we simply assume that $\Phi:I\rightarrow O$ is of the form $$\Phi(x)=\sum_{b,k}j_b E_{bk} x {E_{bk}}^* {j_b}^*\quad\text{for all}\quad x\in I$$ where $E_{bk}:H_{in}\rightarrow H^b_{out}$.
\begin{thm}\label{multigraph_kraus_form}
    Let $\Phi:I\rightarrow O$ be a quantum channel and $\{E_{bk}:H_{in}\rightarrow H_{out}\:|\:b,k\}$ be a family of Kraus operators related to $\Phi$. The quantum confusability multigraph of the channel is given by,
    \begin{equation*}
        \tilde{S}_{\Phi}=span \:\{\sum_{i,a,j,a'} {e^{aa'}_{ij}}\otimes \theta_{(E_{bl}e^{a'}_j), (E_{bk}e^a_i)}\:|\:b,k,l\}\subseteq B(H_{in})\otimes O^{op}.
    \end{equation*}
\end{thm}
\begin{proof}
For each $b$, where choose a Hilbert space $K_b$ with dimension $|\{E_{bk}\:|\:k\}|$, that is the number of Kraus operators corresponding to the index $b$. For a fixed $b$, let $\{s_{bk}\:|\:k\}$ is an orthonormal basis of $K_b$. We define a  Hilbert $O^{op}$ module $\mathcal{E}$ by the following description:
\begin{enumerate}
    \item As a vector space $\mathcal{E}=\bigoplus_b \big(H^b_{out}\otimes K_b\big)$.
    \item The right $O^{op}$ multiplication on $\mathcal{E}$ is given by
    $$ (\sum_{b,k}\xi_{bk}\otimes s_{bk})*T:= (\sum_{b,k}T(\xi_{bk})\otimes s_{bk})$$
    where $T\in O^{op}$ and $\xi_{bk}\in H^b_{out}$.
    \item The $O^{op}$ valued inner product on $\mathcal{E}$ is given by,
    $$\langle \sum_{b,k}\xi_{bk}\otimes s_{bk}, \sum_{b,k}\eta_{bk}\otimes s_{bk}\rangle_{O^{op}}:=\sum_{b,k}\theta_{\eta_{bk},\xi_{bk}}$$
    where $\theta_{\eta_{bk},\xi_{bk}}$'s are rank one operators on $H_{out}$ described in notation \ref{rank-1}.
\end{enumerate}
We define an $O^{op}$-linear map $W:\mathcal{M}_{\Phi}\rightarrow \mathcal{E}$ by,
\begin{align*}
 W(\xi\otimes T)= \sum_{b,k} T E_{bk}\xi \otimes s_{bk}.    
\end{align*}
\textbf{Claim:} The pair $(\mathcal{E}, W)$ is a Stinespring module related to $\Phi$.\\
\textbf{Proof of claim:}
To show that the pair $(\mathcal{E},W)$ is a Stinespring module, we need to show that $W^*W=C_{\Phi}$.
As an element of $B(H_{in})\otimes O^{op}$, using equation \ref{Choi_mat_formula} and identities in notation \ref{rank-1} we observe that,
\begin{align}\label{Choi_mat_kraus_form}
 C_{\Phi}=\sum_{i,j,a}e^{a}_{ij}\otimes \Phi(e^{a}_{ji})=\sum_{i,j,a}e^{a}_{ij}\otimes (\:\sum_{b,k}\theta_{(E_{bk}e^{a}_{j}), (E_{bk}e^{a}_{i})}\:).   
\end{align}
We investigate how the map $W^*:\mathcal{E}\rightarrow \mathcal{M}_{\Phi}$ looks like: For any $\Gamma\in \mathcal{E}$ and $\xi\in H_{in}$,
\begin{align*}
\langle \xi\otimes 1,\: W^*(\Gamma) \rangle_{O^{op}}= \langle \sum_{b,k} E_{bk}(\xi)\otimes s_{bk},\: \Gamma\rangle_{O^{op}}=\sum_{b,k}\theta_{(\Gamma_{bk}), (E_{bk}\xi)}
\end{align*}
where $\Gamma=\sum_{b,k} \Gamma_{bk}\otimes s_{bk}$. Taking $\xi=e^a_i$, we get
\begin{align}\label{W*_form}
W^*(\Gamma)=\sum_{i,a}e^a_i\otimes (\:\sum_{b,k}\theta_{(\Gamma_{bk}), (E_{bk}e^a_i)}\:).
 \end{align}
Therefore,
\begin{align*}
W^*W(e^a_j\otimes 1)=W^*(\sum_{b,k}E_{bk}e^a_j\otimes s_{bk})&=\sum_{b,k}(\:\sum_{i,a'}e^{a'}_i\otimes \theta_{(E_{bk}e^a_j),(E_{bk}e^{a'}_i)}\:)\\
&=\sum_{i}e^a_i\otimes (\:\sum_{b,k}\theta_{(E_{bk}e^a_j),(E_{bk}e^a_i)}\:)
\end{align*}
Hence, as an element of $B(H_{in})\otimes O^{op},\:\:\: W^*W=C_{\Phi}$. Here we have exploited the block form of Kraus operators, that is, for each $E_{bk}$, $Ker(E_{bk})$ contains $(H^{a'}_{in})^{\perp}\subseteq H_{in}$ for some $a'$.\\
It follows that,
\begin{align*}
   \mathcal{L}(\mathcal{E})=\bigoplus_b (id_b\otimes B(K_b))=\bigoplus_b span\:\{id_b\otimes f^b_{kl}\:|\:k,l\} 
\end{align*}
 where $id_b:H^b_{out}\rightarrow H^b_{out}$ is the identity map and $f^b_{kl}\in B(K_b)$ is the elementary matrix which sends $s_{bl}$ to $s_{bk}$.
For $\xi\otimes 1 \in \mathcal{M}_{\Phi}$, using \ref{W*_form}, we observe that,
\begin{align*}
W^*(id_b\otimes f^b_{kl})W(\xi\otimes 1)&=W^*(id_b\otimes f^b_{kl})(\sum_{b',k'}E_{b'k'}\xi\otimes s_{b'k'})\\
&=W^*(E_{bl}\xi\otimes s_{bk})\\
&=\sum_{i,a}e^a_i \otimes \theta_{(E_{bl}\xi),(E_{bk}e^a_i)}.
\end{align*}
By taking $\xi=e^{a'}_j$, we get $$W^*(id_b\otimes f^b_{kl})W(e^{a'}_j\otimes 1)=\sum_{i,a}e^a_i \otimes \theta_{(E_{bl}e^{a'}_j),(E_{bk}e^a_i)}.$$
Therefore the quantum confusability multigraph related to the quantum channel $\Phi$ is given by,
\begin{align*}
    \tilde{S}_{\Phi}=span\:\{\sum_{i,a,j,a'}e^{aa'}_{ij}\otimes \theta_{(E_{bl}e^{a'}_j),(E_{bk}e^a_i)}\:|\: b,k,l\}\subseteq B(H_{in})\otimes O^{op}.
\end{align*}
\end{proof}
 
\begin{rem}
Theorem \ref{multigraph_kraus_form} has the following consequences:
\begin{enumerate}
    \item The confusability multigraph $\tilde{S}_{\Phi}$ is an $(I'\otimes 1)-(I'\otimes 1)$ bimodule where $I'$ is the commutator of $I$ in $B(H_{in})$.
    \item As $O=\bigoplus_b B(H^b_{out})$, we write $$B(H_{in})\otimes O^{op}=\bigoplus_b B(H_{in})\otimes B(H^b_{out})^{op},\quad \Phi_b(x):={j_b}^*\Phi(x)j_b$$
where $j_b:H^b_{out}\rightarrow H_{out}$  is the canonical inclusion map. The quantum confusability multigraph $\tilde{S}_{\Phi}$ admits the following decomposition:
\begin{align*}
\tilde{S}_{\Phi}=\bigoplus_b\: \tilde{S}_{\Phi_b}
\end{align*}
\item It is evident that the quantum confusability multigraph has an algebra (possibly non-unital) structure, that is, $$(\tilde{S}_{\Phi})^2\subseteq (\tilde{S}_{\Phi})\quad \text{and} \quad \tilde{S}^*_{\Phi}=\tilde{S}_{\Phi}.$$ \par
In classical scenarios, that is, in the context of a classical information channel, this algebra structure conveys that, if two inputs $x_1$ and $x_2$ are confused through an output $y$, and same happens for another pair $x_2$ and $x_3$, then $x_1$ and $x_3$ are bound to be confused through the same output $y$. Later in section \ref{decomp_mult}, we will see that any \textbf{symmetric decomposable multi-relation or multigraph} admits this transitive structure and there is a completely positive map whose confusability multigraph coincides with it.
\end{enumerate}

\end{rem}
We apply theorem \ref{multigraph_kraus_form} in two special cases:\\
\textbf{The classical case: $I=l^{\infty} (X)$ and $O=l^\infty (Y)$:}
A classical information channel is given by the following map:
\begin{align*}
    \Phi(e_{xx})=\sum_{y\in Y}p(y\:|\:x) f_{yy}
\end{align*}
where $p(y|x)$ is the probability getting the output $y$ upon receiving input $x$. The Kraus operators $E_{xy}:l^{2} (X)\rightarrow span\{e_{y}\}$ is a rank one operator given by, 
\begin{align*}
E_{xy} (e_{x'}) = \delta_{xx'} p(y|x)^{\frac{1}{2}} f_{y} \quad\text{for all}\quad x\in X
\end{align*}
Therefore the quantum confusability multigraph $\tilde{S}_{\Phi}\subseteq B(l^2(X))\otimes l^{\infty}(Y)$ is given by,
\begin{align*}
\tilde{S}_{\Phi}&=span\:\{e_{x_1x_2}\otimes \theta_{(E_{xy} e_{x_1}), (E_{x'y} e_{x_2})}\:|\:x_1,x_2,x, x' \in X, y\in Y\}\\&
= span\: \{e_{x_1x_2}\otimes p(y|x_1)^{\frac{1}{2}}p(y|x_2)^{\frac{1}{2}} f_{yy}\:|\:x_1,x_2\in X, y\in Y\}\\&
=span\: \{e_{x_1x_2}\otimes f_{yy}\:|\:x_1,x_2\in X, y\in Y\:\:\text{such that}\:\:p(y|x_1)p(y|x_2)\neq 0\}
\end{align*}

\textbf{The purely quantum case:  $I=B(H_{in}), O=B(H_{out})$:} 
The quantum confusability multigraph $\tilde{S}_{\Phi}$ is given by,
\begin{align*}
    \tilde{S}_{\Phi}=span\:\{\sum_{i,j} e_{ij}\otimes \theta_{(E_le_j),(E_ke_i)}\:|\:k,l\}\subseteq B(H_{in})\otimes B(H_{out})^{op}
\end{align*}
where $E_k:H_{in}\rightarrow H_{out}$'s are Kraus operators related to $\Phi$. Under the isomorphism $(id\otimes \pi_{op})$, $B(H_{in})\otimes B(H_{out})^{op}$ is isomorphic to $B(H_{in})\otimes B(\overline{H_{out}})$. It further follows that, 
\begin{align*}
    (id\otimes \pi_{op})(\tilde{S}_{\Phi})=span\:\{\sum_{i,j}e_{ij}\otimes \theta_{(\overline{E_ke_i}),(\overline{E_le_j})}\:|\:k,l\}\subseteq B(H_{in})\otimes B(\overline{H_{out}})
\end{align*}
Before proceeding further, we introduce the $vec$ and $mat$ functions.
\begin{nota}
    Let $H$ and $K$ be finite dimensional Hilbert spaces. We define $vec:B(K,H)\rightarrow H\otimes \overline{K}$ and $mat:H\otimes \overline{K}\rightarrow B(K,H)$ via the following description: for $\xi\in H$ and $\eta\in K$, 
    \begin{align*}
    vec(\theta_{\xi,\eta})=\xi\otimes \overline{\eta}\quad \text{and}\quad mat(\xi\otimes \overline{\eta})=\theta_{\xi,\eta}.
    \end{align*}
\end{nota}
Using these Hilbert space isomorphisms, we observe that, 
\begin{align*}
    B(H_{in})\otimes B(\overline{H_{out}}) \cong H_{in}\otimes \overline{H_{in}} \otimes \overline{H_{out}}\otimes H_{out}
   &\cong H_{in}\otimes \overline{H_{out}}\otimes \overline{H_{in}} \otimes  H_{out}
\end{align*}
All the above isomorphisms are isomorphisms of Hilbert spaces. Operator spaces are considered as Hilbert spaces with respect to Hilbert Schmidt inner product. Under this transformation, The quantum confusability multigraph is described as follows:
\begin{align}\label{quant_mult_kraus_form_2}
    (id\otimes \pi_{op})\tilde{S}_{\Phi}\cong span\:\{vec({E_k}^*)\otimes \overline{vec({E_l}^*)}\:|\:k,l\}.
\end{align}
We observe the following:
\begin{prop}\label{quant_confus_mult_alg_str_prop}
 Let $\Phi:B(H_{in})\rightarrow B(H_{out})$ be a quantum channel. Then following isomorphism of vector spaces holds:
 \begin{align*}
     (id\otimes \pi_{op})(\tilde{S}_{\Phi})\cong V\otimes \overline{V}
 \end{align*}
 where $V=Range
 \: \big((id\otimes \pi_{op})(C_{\Phi})\big)$ and $(id\otimes \pi_{op})(C_{\Phi})$ is seen as linear operator on $H_{in}\otimes \overline{H_{out}}$.
\end{prop}
\begin{proof}
 Any linear map $\Phi:B(H_{in})\rightarrow B(H_{out})$ can be recovered from its Choi-type invariant $C_{\Phi}\in B(H_{in})\otimes B(H_{out})^{op}$ using Choi-Jamilkowski representation. For $e_{i'j'}\in B(H_{in})$, we observe that,
    \begin{align}
     (tr\otimes id)\big( (e_{i'j'}\otimes 1)C_{\Phi})&=\sum_{i,j}(tr\otimes id)\big((e_{i'j'}\otimes 1)(e_{ij}\otimes \Phi(e_{ji}))\big) \notag \\
     &=\sum_{j}(tr\otimes id) \big( e_{i'j}\otimes \Phi(e_{jj'}) \big)=\Phi(e_{i'j'})\label{choi_jamil_rep}
    \end{align}
    We note that the element $\Phi(e_{i'j'})$ in the above expression lies in $B(H_{out})^{op}$ which is the same vector space as $B(H_{out})$. Moreover, proposition \ref{CP_module_correspondence} and the following remark ensure that $\Phi$ is CP iff $C_{\Phi}$ is a positive element in $B(H_{in})\otimes B(H_{out})^{op}$. \par
    Using spectral theorem for positive operators, we observe that,
    \begin{align*}
       (id\otimes \pi_{op}) (C_{\Phi})=\sum_{\alpha} \theta_{\Gamma_\alpha, \Gamma_{\alpha}}
    \end{align*}
    where $\{\Gamma_{\alpha}\:|\:\alpha\}\in H_{in}\otimes \overline{H_{out}}$ is a set of mutually orthogonal vectors and $$span\:\{\Gamma_{\alpha}\:|\:\alpha\}=Range\:\big((id\otimes \pi_{op})C_{\Phi}\big).$$
    It follows that, for $e_{i'j'}\in B(H_{in})$,
    \begin{align}
         (tr\otimes id)\big((e_{i'j'}\otimes 1)\sum_{\alpha}\theta_{\Gamma_{\alpha},\Gamma_{\alpha}}\big)&= \sum_{k,\alpha}(tr\otimes id) \big( \theta_{(e_{i'}\otimes \overline{f_k}), (e_{j'}\otimes \overline{f_k})}(\theta_{\Gamma_{\alpha},\Gamma_{\alpha}})\big) \notag\\
         &=\sum_{k,\alpha,k'} \overline{(\Gamma_{\alpha})_{i'k'}} (\Gamma_{\alpha})_{j'k}  \overline{f_{kk'}}\label{choi_jamil_rep_2}.
    \end{align}
    In the above computation, we have assumed that $\{\overline{f_k}\:|\:k\}$ is an orthonormal basis of $\overline{H_{out}}$ and therefore the set $\{\overline{f_{kl}}:=\theta_{\overline{f_k},\overline{f_l}}\:|\:k,l\}$ is a canonical basis of $B(\overline{H_{out}})$ consisting of elementary matrices.\\
    From equations $\ref{choi_jamil_rep}$ and \ref{choi_jamil_rep_2}, it follows that,
    \begin{align*}
    \Phi(e_{i'j'})&=\pi^{-1}_{op}\bigg( (tr\otimes id)\big((e_{i'j'}\otimes 1)\sum_{\alpha}\theta_{\Gamma_{\alpha},\Gamma_{\alpha}}\big)\bigg)\\
    &=\sum_{k,\alpha,k'} \overline{(\Gamma_{\alpha})_{i'k'}} (\Gamma_{\alpha})_{j'k}  f_{k'k}=\sum_{\alpha}\theta_{(\Lambda_{\alpha} e_{i'}),(\Lambda_{\alpha} e_{j'})}=\sum_{\alpha} \Lambda_{\alpha} (e_{i'j'}) \Lambda^*_{\alpha}  
 \end{align*}
 where $\Lambda_{\alpha}:H_{in}\rightarrow H_{out}$'s are linear operators defined by $\Lambda_{\alpha}=mat(\Gamma_{\alpha})^*$. As the family $\{\Lambda_{\alpha}\:|\:\alpha\}$ gives Kraus decomposition for the CP map $\Phi$, using equation \ref{quant_mult_kraus_form_2}, we observe that, 
 \begin{align*}
    (id\otimes \pi_{op}) \tilde{S}_{\Phi}&=span\:\{vec(\Lambda^*_{\alpha})\otimes \overline{vec({\Lambda^*_{\beta}})}\:|\:\alpha,\beta\}\\&
    =span\:\{\Gamma_{\alpha}\otimes \overline{\Gamma_{\beta}}\:|\:\alpha,\beta\}=V\otimes \overline{V}
 \end{align*}
 where $V=span\:\{\Gamma_{\alpha}\:|\:\alpha\}=Range\big((id\otimes \pi_{op})C_{\Phi}\big)$.
\end{proof}

We conclude this section with the following remark.

\begin{rem}\label{decomp_rem}
    The Hilbert space $\overline{H_{in}}\otimes H_{out}$ can be identified with $B(H_{in}, H_{out})$ via the map $\overline{\xi}\otimes \eta \mapsto \theta_{\eta, \xi}$ where $\xi\in H_{in}$ and $\eta\in H_{out}$. Similarly, there is a Hilbert space isomorphism between $H_{in}\otimes \overline{H_{out}}$ and $B(H_{out}, H_{in})$. Under these Hilbert space isomorphisms, using proposition \ref{quant_confus_mult_alg_str_prop}, we get
    \begin{align*}
         (id\otimes \pi_{op})\tilde{S}_{\Phi}\cong \mathcal{K}^* \otimes \mathcal{K}
    \end{align*}
    where $\mathcal{K}=span\{{E_k}\:|\:k\}$ is the Kraus space of the quantum channel $\Phi$.
\end{rem}
\section{\textbf{Quantum Multi-relations}}\label{quant_mult}
Let $X$ and $Y$ be two finite sets. A \textbf{multi-relation} $R$ on this pair $(X,Y)$ given by a subset
\begin{align*}
    R\subseteq X\times X\times Y.
\end{align*}
For $(x_1, x_2, y)\in R$, we say that ``$x_1$ and $x_2$ are related through $y$".\par 
In the spirit of Weaver's work on quantum relations (\cite{MR2908249}) and the discussions in previous section, we describe ``quantum multi-relations" on a pair of finite dimensional algebras $(M,N)$. 

\begin{defn}\label{defn_multirelation}
    Let $M\subseteq B(H)$ and $N\subseteq B(K)$ be two finite dimensional algebras, A \textbf{quantum multi-relation} or a \textbf{multigraph} on the pair $(M,N)$ is given by a subspace $V\subseteq B(H\otimes K)$ such that following conditions hold:
    \begin{enumerate}
    \item $V\subseteq B(H) \otimes N$.
        \item $V$ is an $(M'\otimes 1)-(M'\otimes 1)$ bimodule in $B(H\otimes K)$.
        \item $(1\otimes Z(N))(V)\subseteq V$ where $Z(N)=N\cap N'$, that is, the center of $N$.
\end{enumerate}
A multi-relation $V$ on $(M,N)$  is called \textbf{symmetric} if $V^*=V$ and \textbf{transitive} if $V^2\subseteq V$. 
\end{defn}
 Proposition \ref{classical_multi} establishes a bijection between quantum multi-relations and classical multi-relations on a pair of finite sets.
\begin{prop}\label{classical_multi}
Given a pair of finite sets $(X,Y)$ we consider the von Neumann algebras $l^{\infty}(X)\subseteq B(l^2(X))$ and $l^{\infty}(Y)\subseteq B(l^2(Y))$. Given a multi-relation $R\subseteq X\times X\times Y$, we define $V_{R}\subseteq B(l^2(X))\otimes l^{\infty}(Y)$ by,
\begin{align*}
V_{R}=\:span\:
\{e_{x_1x_2}\otimes e_{yy}\:|\:(x_1,x_2,y)\in R\}
\end{align*}
Conversely, for a quantum multi-relation $V\subseteq B(l^2(X))\otimes l^{\infty}(Y)$ we define $R_V\subseteq X\times X\times Y$ by,
\begin{align*}
R_V=\{(x_1,x_2,y)\:|\: \exists\: T\in V\:\:\text{such that}\:\: \langle e_{x_1}\otimes e_y,T(e_{x_2}\otimes e_y)\rangle\neq 0 \}
\end{align*}
These assignments are inverses of each other and establish a bijective correspondence between quantum multi-relations on the pair of algebras $(l^{\infty}(X),l^{\infty}(Y))$ and multi-relations on the pair of sets $(X,Y)$.
\end{prop}
\begin{proof}
It is easy to see that for a multi-relation $R\subseteq X\times X\times Y$, $V_R$ is a quantum multi-relation on $(l^{\infty}(X), l^{\infty}(Y))$. Conversely, for any quantum multi-relation $V\subseteq B(l^2(X))\otimes l^{\infty}(Y)$, any $T\in V$ can be written as
\begin{align}\label{classical_T_decomp_matrix_elements}
T=\sum_{\substack{x_1,x_2\in X\\y\in Y}} \langle e_{x_1}\otimes e_y, T(e_{x_2}\otimes e_y)\rangle (e_{x_1x_2}\otimes e_{yy})
\end{align}
As $V$ satisfies (2) and (3) of definition \ref{defn_multirelation}, for any $x_1,x_2\in X, y\in Y$, 
\begin{align*}
    (e_{x_1x_1}\otimes e_{yy})T(e_{x_2x_2}\otimes 1)=\langle e_{x_1}\otimes e_y, T(e_{x_2}\otimes e_y)\rangle (e_{x_1x_2}\otimes e_{yy})\in V.
\end{align*}
If $\langle e_{x_1}\otimes e_y, T(e_{x_2}\otimes e_y)\rangle\geq 0$, for some $x_1, x_2$ and $y$, then $e_{x_1x_2}\otimes e_{yy}\in V$. We further conclude that $V_{R_V}\subseteq V$. As any $T\in V$ is of the form \ref{classical_T_decomp_matrix_elements}, we conclude that $V_{R_V}=V$. \par
Given a multi-relation $R\subseteq X\times X\times Y$, it is easy to see that $R_{V_{R}}=R$. Therefore the bijective correspondence is established.
\end{proof}

Theorem \ref{embedding_indep} suggests that multi-relations on a pair of finite dimensional algebras $(M,N)$ are representation independent.
\begin{thm}\label{embedding_indep}
Let $M_i\subseteq B(H_i)$ and $N_i\subseteq B(K_i)$, $i=1,2$ be two pairs of finite dimensional algebras such that $M_1\cong M_2$ and $N_1\cong N_2$. There is a $1-1$ correspondence between quantum multi-relations on the pairs $(M_1, N_1)$ and $(M_2,N_2)$.
\end{thm}
\begin{proof}
The proof is similar to the proof of theorem 2.7 in \cite{MR2908249}. We know that any finite dimensional algebra $M$ is  direct sum of full matrix algebras and any representation of $M$ on a finite dimensional Hilbert space $H$ is unitarily equivalent to direct sum of amplifications of its matrix summands in $B(H)$ (consequence of Theorem IV, 5.5, \cite{MR548728} for finite dimensional algebras).      \par 
Therefore, without loss of generality, we can assume that
\begin{align*}
H_2=\tilde{H}\otimes H_1,\quad K_2=\tilde{K}\otimes K_1,\quad M_2=id_{\tilde{H}}\otimes M_1 \quad\text{and}\quad N_2=id_{\tilde{K}}\otimes N_1.
\end{align*}
We also have the following equalities:
\begin{align}\label{commutator_identities_indep_proof}
M'_2=B(\tilde{H})\otimes M'_1\quad Z(N_2)=id_{\tilde{K}}\otimes Z(N_1)
\end{align}
Let us write $B(\tilde{H})=span\{e_{ij}\:|\:i,j\}$ where $e_i$'s form an orthonormal basis of $\tilde{H}$ and $e_{ij}$'s are elementary matrix units.
Let $\mathcal{F}_i$ be the families of quantum multi-relations over the pairs of algebras $(M_i,N_i)$. We define $\Phi:\mathcal{F}_1\rightarrow \mathcal{F}_2$ and $\Psi:\mathcal{F}_2\rightarrow \mathcal{F}_1$ by
\begin{align*}
\Phi(V)=B(\tilde{H})\otimes id_{\tilde{K}}\otimes V\quad\text{and}\quad \Psi(\tilde{V})=\{A\in B(H_1\otimes K_1)\:|\: e_{i_0i_0}\otimes id_{\tilde{K}}\otimes A\in \tilde{V}\}
\end{align*}

where $e_{i_0i_0}$ is a fixed rank one projection in $B(\tilde{H})$. We need to check the following:
\begin{enumerate}
\item $\Phi$ and $\Psi$ are well defined. 
\item $\Phi$ and $\Psi$ are inverses of each other.
\end{enumerate}
Using identities \ref{commutator_identities_indep_proof} it is easy to see that these maps are well defined and also $\Psi \circ \Phi=id_{\mathcal{F}_1}$. To show that $\Phi\circ\Psi=id_{\mathcal{F}_2}$, it is enough to show that
\begin{align*}
\tilde{V}=B(\tilde{H})\otimes id_{\tilde{K}}\otimes \Psi(\tilde{V}).
\end{align*}
Let us define $\tilde{W}$ to be the right hand side of the above equation.
For any $e_{ij}\in B(\tilde{H})$, we observe that, $$(e_{ij}\otimes id_{\tilde{K}}\otimes A)=(e_{ii_0}\otimes id_{\tilde{K}}\otimes 1)(e_{i_0i_0}\otimes id_{\tilde{K}}\otimes A)(e_{i_0j}\otimes id_{\tilde{K}}\otimes 1)\in \tilde{V}.$$
Therefore $\tilde{W}\subseteq \tilde{V}$. To show $\tilde{V}\subseteq \tilde{W}$, we observe that, for any $\tilde{A}\in \tilde{V}$, it is of the form: $\tilde{A}=\sum_{i,j}e_{ij}\otimes id_{\tilde{K}}\otimes \tilde{A}_{ij}$. Therefore,  to prove the claim, it is enough to show that $(e_{ij}\otimes id_{\tilde{K}}\otimes \tilde{A}_{ij})\in \tilde{W}$ for all $i,j$. We observe the following identities:
\begin{align*}
 e_{ij}\otimes id_{\tilde{K}}\otimes A_{ij}&=( e_{ii}\otimes id_{\tilde{K}}\otimes id_{{H_1\otimes K_1}})( \sum_{i',j'}e_{i'j'}\otimes id_{\tilde{K}}\otimes A_{i'j'})(e_{jj}\otimes id_{\tilde{K}}\otimes id_{{H_1\otimes K_1}})\in \tilde{V};\\
 e_{i_0i_0}\otimes id_{\tilde{K}}\otimes A_{ij}&=( e_{i_0i}\otimes id_{\tilde{K}}\otimes id_{H_1\otimes k_1})(e_{ij}\otimes id_{\tilde{K}}\otimes A_{ij})(e_{ji_0}\otimes id_{\tilde{K}}\otimes id_{H_1\otimes K_1})\in \tilde{V}
\end{align*}
Therefore $A_{ij}\in \Psi (\tilde{V})$ and hence $e_{ij}\otimes id_{\tilde{K}}\otimes {A_{ij}}\in \tilde{W}$ completing the proof.
\end{proof}
\begin{rem}
    It is not hard to see that definition \ref{defn_multirelation} and theorem \ref{embedding_indep} can be described for general von Neumann algebras considering weak* closed subspaces and spatial tensor products. As we are more interested in interplay of ``multigraphs" and its underlying simple or single-edged graphs which is obtained by applying trace on the quantum label set $N$, we have limited our treatments to finite dimensional structures in this article.  
\end{rem}
\subsection*{Weaver action and multi-edge indicator:} 
Though theorem \ref{embedding_indep} suggests that the multi-relations on a pair of algebras $(M,N)$ are in bijective correspondence if $M$ and $N$ are represented on different Hilbert spaces, we will subject ourselves to use a ``minimal" representation for $N$ for the rest of this section. For $N\subseteq B(K)$, we will say that $N$ is ``minimally represented" on $K$ if $Z(N)=N'\subseteq B(K)$. As $N$ is isomorphic to direct sum of matrix algebras, such representation always exists and any two minimal representations of $N$ are unitarily equivalent.\par 
Before proceeding further,  we briefly recall what Weaver action is (see proposition 2.23 of \cite{MR2908249} or section 5.1 in  \cite{MR4706978} ) in the context of a quantum relations. For $M\subseteq B(H)$, we define the Weaver action of $M\otimes M^{op}$ on $B(H)$ by,
\begin{align*}
    \pi(T_1\otimes T_2)(S)=T_1S T_2\quad\text{where}\quad T_1\otimes T_2\in M\otimes M^{op},\:\:S\in B(H).
\end{align*}

If $V\subseteq B(H)$ is an $M'-M'$ bimodule, then we define the annihilator $\mathcal{I}_V$ of the subspace $V$  by,
\begin{align*}
\mathcal{I}_V= \{\mathcal{B}\in M\otimes M^{op}\:|\: \pi(\mathcal{B})(T)=0\quad\text{for all}\quad T\in V\}
\end{align*}
Equipping $B(H)$ with a Hilbert Schmidt inner product coming from trace we see that $B(H)\cong H\otimes \overline{H}$ as Hilbert spaces. Under this isomorphism,  any $M'-M'$ bimodule in $B(H)$ is $(M\otimes M^{op})'$ left module in $H\otimes\overline{H}$. The ideal $\mathcal{I}_V$ is also the annihilator of $V$ under the left action of $M\otimes M^{op}$ on $H\otimes \overline{H}$. Now we will use the following lemma:
\begin{lem}\label{left_action_annhi_proj}
Let $M\subseteq B(H)$  be a von Neumann algebra and $H$ is not necessarily finite dimensional. For a subspace $V$ in $H$ which is also a left $M'$ module, the annihilator of $V$ under the action of $M$ is given by
\begin{align*}
    \mathcal{I}_V=M\tilde{\mathcal{P}}_V
\end{align*}
where $\tilde{\mathcal{P}}_V$ is orthogonal projection onto $V^{\perp}$.
\end{lem}
\begin{proof}
For any self adjoint $T\in M'$,
\begin{align*}
  T\tilde{\mathcal{P}}_V= \tilde{\mathcal{P}}_V T\tilde{\mathcal{P}}_V= (\tilde{\mathcal{P}}_V T\tilde{\mathcal{P}}_V)^*=\tilde{\mathcal{P}}_V  T
\end{align*}
as $T(V)\subseteq V$ and $T(V^{\perp})\subseteq V^{\perp}$. Therefore $T$ commutes with $\tilde{\mathcal{P}}_V $ for all $T\in M'$. Hence, $\tilde{\mathcal{P}}_V \in M''=M$ as $M$ is a von Neumann algebra. As $Ker(\tilde{\mathcal{P}}_V)=V$, it follows that $\tilde{\mathcal{P}}_V\in \mathcal{I}_V$ and therefore $M\tilde{\mathcal{P}}_V\subseteq \mathcal{I}_V$. Conversely, for $T\in \mathcal{I}_V$ and $\xi\in H$, 
\begin{align*}
T\xi= T(1-\tilde{P}_V)\xi+T\tilde{P}_V\xi=T\tilde{\mathcal{P}}_V\xi
\end{align*}
Therefore $T\in M\tilde{\mathcal{P}}_V$.
\end{proof}
The discussion prior to lemma \ref{left_action_annhi_proj} along with lemma \ref{left_action_annhi_proj} suggest that, for any $M'-M'$ bimodule $V$ in $B(H)$, the generator of $\mathcal{I}_V$ under the Weaver action is given by $\tilde{\mathcal{P}}_V$ where $\tilde{\mathcal{P}}_V$ is orthogonal projection onto $V^{\perp}\subseteq  B(H)$.\par 
Now we move to the context of quantum multi-relations on a pair of finite dimensional algebras $(M,N)$. As we are working with a minimal representation for $N$, we can argue that a quantum multi-relation on $(M,N)$ is a subspace of $B(H)\otimes N\subseteq B(H\otimes K)$ which is a $(M\otimes N)' - (M\otimes N)'$ bimodule. Therefore using Weaver action of $M\otimes N$ on $B(H\otimes K)$, we get a projection $\tilde{\mathcal{P}}_V$ in $M\otimes M^{op}\otimes N\otimes N^{op}$ which annihilates $V$ in $B(H\otimes K)$.

\begin{defn}
    For a quantum multi-relation $V\subseteq B(H)\otimes N$ on a pair of finite dimensional algebras $(M,N)$ the \textbf{quantum multi-edge indicator} is defined by
    \begin{align*}
    \mathcal{P}_V=1-\tilde{\mathcal{P}_V}
    \end{align*}
    where $\tilde{\mathcal{P}_V}$ is the generator of the annihilator ideal $\mathcal{I}_V$ in $M\otimes M^{op}\otimes N\otimes N^{op}$ under the Weaver action on $B(H)\otimes B(K)$.
\end{defn}

\begin{rem}
Any projection in $M\otimes M^{op}\otimes N\otimes N^{op}$ does not produce a quantum multi-edge indicator unlike the case of quantum relations. Projections with range lying within $B(H)\otimes N$ only serve as  quantum multi-edge indicators of multi-relations on $(M,N)$. 
\end{rem}
\begin{defn}
Given a quantum multigraph $V\subseteq B(H)\otimes N$, the \textbf{underlying  quantum single edged graph} is define by 
\begin{align*}
    \tilde{V}=(id\otimes tr_K)(V). 
\end{align*}

\end{defn}
For a classical multigraph $(V,E,s,t)$, the underlying single edged graph is achieved by removing all edges between two fixed points and replacing them with a single edge between them.\par 
\begin{prop}\label{P_v-S_v_relation}
    Let $V\subseteq B(H)\otimes N$ be a quantum multigraph on a pair of algebras $(M,N)$ and $\tilde{V}$ be the underlying quantum single edged graph on $M$. Let us define $\mathcal{S}_V\in M\otimes M^{op}$ by 
    \begin{align*}
        \mathcal{S}_V=(id\otimes id\otimes tr_K)(id\otimes id\otimes m)(\mathcal{P}_V)
    \end{align*}
    where $m:N\otimes N^{op}\rightarrow N$ is defined by $m(a\otimes b)=ab$, that is the multiplication in $N$ and $\mathcal{P}_{V}$ is the quantum multi-edge indicator for $V$. It follows that $\mathcal{S}_V$ is a positive element in $M\otimes M^{op}$ with range $\tilde{V}$ under the Weaver action of $M\otimes M^{op}$ on $B(H)$.
\end{prop}
\begin{proof}
Let us define $i:B(H)\rightarrow B(H)\otimes N$ by $i(T)=T\otimes 1$. Equipping $B(H)\otimes N$ with Hilbert Schmidt inner product coming from $tr_H\otimes tr_K$, we observe that, 
\begin{align*}
   \langle i^*(T\otimes S),T'\rangle_{tr_H}= \langle T\otimes S,T'\otimes 1\rangle_{tr_H\otimes tr_K}=\langle T, T'\rangle_{tr_H} tr_K(S).
\end{align*}
Therefore $i^*(\mathcal{T})=(id\otimes tr_K)(\mathcal{T})$ for all $\mathcal{T}\in B(H)\otimes N$. As $\pi(\mathcal{P}_V)$ projects $B(H)\otimes N$ onto $V$, it is easy to see that the diagram in figure \ref{weighted_edge_indi_diag} commutes.
\begin{figure}[h]
\begin{tikzcd}[row sep=2.5em, column sep=3.5em]
B(H)\otimes N 
    \arrow[r, "\pi(\mathcal{P}_V)"]
&
V 
    \arrow[d, "\: i^{*}=(\mathrm{id}\otimes \mathrm{tr}_K) \:"]
\\
B(H) 
    \arrow[u, "i\:"]
    \arrow[r, "i^{*} \pi(\mathcal{P}_V) i"']
&
\widetilde{V}
\end{tikzcd}
\caption{}\label{weighted_edge_indi_diag}
\end{figure}
Let us assume that $\mathcal{P}_V=\sum P_1\otimes P_2\otimes  P_3\otimes P_4$. Then for any $T\in B(H)$,
\begin{align*}
    i^* \pi(\mathcal{P}_V)\:i(T)=i^*(\sum P_1 T P_2 \otimes P_3 P_4)& =(\sum P_1 T P_2)tr_K(P_3P_4)\\
    &=\pi\big((id\otimes id\otimes tr_K)(id\otimes id\otimes m)(\mathcal{P}_V)\big)(T).
\end{align*}
Therefore $Range(\pi(\mathcal{S}_V))\subseteq \tilde{V}$ under . Let us define, 
\begin{align*}
    W=B(H)\otimes 1\subseteq B(H)\otimes N
\end{align*}
and $\mathcal{P}_W$ be the orthogonal projection from $B(H)\otimes N$ onto $W$. We observe that, 
\begin{align*}
Range(i^*\pi(\mathcal{P}_V)i)= Range(\mathcal{P}_W\: \pi(\mathcal{P}_V)\:\mathcal{P}_W) \quad\text{and}\quad \tilde{V}=Range(\mathcal{P}_W\:\pi(\mathcal{P}_V)).
\end{align*}
As for any operator $T$ on a finite dimensional Hilbert space, $Range (T)= Range (TT^*)$, it follows that, 
\begin{align*}
    Range(\pi(\mathcal{S}_V))= Range(\mathcal{P}_W\: \pi(\mathcal{P}_V)\:\mathcal{P}_W)=Range(\mathcal{P}_W\:\pi(\mathcal{P}_V))=\tilde{V}.
\end{align*}
\end{proof}
\begin{rem}\label{weighted_edge_indicator_rem}
The element $\mathcal{S}_V$ will be called the \textbf{edge indicator of the weighted underlying single edged graph} of a multigraph $V$. 
\end{rem}
\subsection*{Quantum multi-adjacency operators:}
We introduce different quantum adjacency operators related to a quantum multi-relation $V\subseteq B(H)\otimes N$ on a pair of algebras $(M,N)$. To describe these operators in a way consistent with the existing theory of the quantum ``single-edged" graphs, we will assume that $M$ and $N$ both are minimally represented on $H$ and $K$ respectively, that is, $Z(M)=M'\subseteq B(H)$ and $Z(N)=N'\subseteq B(K)$.   
\begin{defn}\label{multi_adjacency_def}
    For a quantum multi-relation $V\subseteq  B(H)\otimes N$, the \textbf{quantum multi-adjacency operator} $\mathcal{A}:M\otimes N\rightarrow M\otimes N$ is defined by,
    \begin{align*}
        \mathcal{A}_{\mathcal{P}_V}(m\otimes n)=(tr_{H}\otimes id\otimes tr_{K}\otimes id)(\mathcal{P}_V (m\otimes 1\otimes n\otimes 1))
    \end{align*}
The weighted adjacency operator $\mathcal{A}_{\mathcal{S}_V}:M\rightarrow M$ is given by, 
\begin{align*}
\mathcal{A}_{\mathcal{S}_V}(m)=(tr_H\otimes id)(\mathcal{S}_V(m\otimes 1))
\end{align*}
Adjacency operator of the underlying single edged graph is given by
\begin{align*}
    \mathcal{A}_{\mathcal{P}_{\tilde{V}}}(m)   =(tr_H\otimes id)(\mathcal{P}_{\tilde{V}}(m\otimes 1))
\end{align*}
where $\mathcal{P}_{\tilde{V}}$ is the edge indicator of the underlying single-edged graph $\tilde{V}$.
\end{defn}
We equip $M$ and $N$ with Hilbert Schmidt inner product with respect to $tr_H$ and $tr_K$ and call it $L^2(M)$ and $L^2(N)$ respectively. The algebras $M$ and $N$ can be seen as subalgebras of $B(L^2(M))$ and $B(L^2(N))$ respectively via GNS representations. Given an element $x$ in $M$ or $N$, $\Lambda(x)$ will denote the corresponding element in $L^2(M)$ or $L^2(N)$.
\begin{prop}
 Let $V\subseteq B(H)\otimes N$ be a quantum multi-relation and $\mathcal{A}_{\mathcal{P}_V}$ and $\mathcal{A}_{\mathcal{S}_V}$ be the corresponding adjacency operators defined in definition \ref{multi_adjacency_def}. The following equality holds:
 \begin{align*}
 (id\otimes tr_{L^2(N)})(\mathcal{A}_{\mathcal{P}_V})=\mathcal{A}_{\mathcal{S}_V}.
 \end{align*}
\end{prop}
\begin{proof}
Assuming $\mathcal{P}_V=\sum P_1\otimes P_2\otimes P_3\otimes P_4$, we observe that for $m\in M$, $n\in N$
\begin{align*}
    \mathcal{A}_{\mathcal{P}_V}(m\otimes n)&=\sum tr_H(P_1m)P_2\otimes  tr_K(P_3n)P_4,\\
    \mathcal{A}_{\mathcal{S}_V}(m)&=\sum tr_H(P_1)P_2 tr_K(P_3P_4).
\end{align*}
As $A_{\mathcal{P}_V}\in B(L^2(M))\otimes B(L^2(N))$, and $A_{\mathcal{S}_V}\in B(L^2(M))$, it follows that,
\begin{align}
\mathcal{A}_{\mathcal{P}_V}&=\sum \theta_{\Lambda(P_2),\Lambda({P_1}^*)}\otimes \theta_{\Lambda(P_4),\Lambda({P_3}^*)},\notag\\
\mathcal{A}_{\mathcal{S}_V}&=\sum \theta_{\Lambda(P_2),\Lambda({P_1}^*)} \langle \Lambda({P_3}^*), \Lambda(P_4) \rangle_{L^2(N)}\label{adjacency_operators_rep_lambda}
\end{align}
where $\theta_{\Lambda(P_2),\Lambda({P_1}^*)}$ and $\theta_{\Lambda(P_4),\Lambda({P_3}^*)}$'s are rank one operators in $B(L^2(M))$ and $B(L^2(N))$ respectively. As $tr_{L^2(N)}(\theta_{\Lambda(P_4),\Lambda({P_3}^*)})= \langle \Lambda({P_3}^*), \Lambda(P_4) \rangle_{L^2(N)}$, from equations \ref{adjacency_operators_rep_lambda}, the proposition follows.
\end{proof}
Considering the  algebra multiplication $m$ in $M$ as a linear map $L^2(M)\otimes L^2(M)\rightarrow L^2(M)$ we may consider its adjoint $m^*$. We will use the same notation $m^*$ for different algebras, the underlying algebra will always be clear from the context.
The following theorem suggests that usual properties of adjacency operators hold. 
\begin{thm}\label{multi_adjacency_classification}
    For a multi-relation $V\subseteq B(H)\otimes N$ on a pair of algebras $(M,N)$ the quantum adjacency operators $\mathcal{A}_{\mathcal{P}_V}, \mathcal{A}_{\mathcal{S}_V}$ and $\mathcal{A}_{\mathcal{P}_{\tilde{V}}}$ are completely positive maps. Moreover,  $\mathcal{A}_{\mathcal{P}_V}$ and $\mathcal{A}_{\mathcal{P}_{\tilde{V}}}$ satisfy Schur idempotent condition, that is, 
    \begin{align*}
    m(\mathcal{A}_{\mathcal{P}_V}\otimes \mathcal{A}_{\mathcal{P}_V}) m^*=\mathcal{A}_{\mathcal{P}_V}\quad \text{and}\quad  m(\mathcal{A}_{\mathcal{P}_{\tilde{V}}}\otimes \mathcal{A}_{\mathcal{P}_{\tilde{V}}}) m^*=\mathcal{A}_{\mathcal{P}_{\tilde{V}}}
    \end{align*}
    where $\mathcal{A}_{\mathcal{P}_V}$ and $\mathcal{A}_{\mathcal{P}_{\tilde{V}}}$ are considered as elements of $B(L^2(M)\otimes L^2(N))$ and $B(L^2(M))$ respectively.
\end{thm}
\begin{proof}

As $M$ and $N$ both are finite dimensional algebras and hence isomorphic to direct sum of matrix algebras, we write
\begin{align*}
    M \cong span \{e^a_{ij}\:|\:i,j,a\}, \quad\text{and}\quad N\cong span\{f^b_{kl}\:|\:b,k,l\}.
\end{align*}
Under above identifications, the Choi-type invariants corresponding to the maps $\mathcal{A}_{\mathcal{P}_V}$, $\mathcal{A}_{\mathcal{S_V}}$ and $\mathcal{A}_{\mathcal{P}_{\tilde{V}}}$ are $\mathcal{P}_V$, $\mathcal{S}_V$ and $\mathcal{P}_{\tilde{V}}$ respectively. Since these are all positive operators, it follows that their associated maps completely positive(remark \ref{Choi-type_invar}).  Moreover,  $\mathcal{P}_V$ and $\mathcal{P}_{\tilde{V}}$ are projections, which implies that the associated maps satisfy Schur idempotent conditions. We will provide a short proof for the case of  $\mathcal{A}_{\mathcal{P}_{\tilde{V}}}$, the argument for $\mathcal{A}_{\mathcal{P}_V}$ will follow simply by replacing $M$ with $M\otimes N$.\par  
As $M$ is minimally represented, it follows that, 
\begin{align*}
m^*(e^a_{ij})=\sum_{k} e^a_{ik}\otimes e^a_{kj} 
\end{align*}
As $\mathcal{P}_{\tilde{V}}$ is the Choi-type invariant of $\mathcal{A}_{\mathcal{P}_{\tilde{V}}}$, we observe that,
\begin{align*}
 \mathcal{P}_{\tilde{V}}=\sum_{i,j,a} e^a_{ij}\otimes \mathcal{A}_{\mathcal{P}_{\tilde{V}}}(e^a_{ji}) \quad\text{and}\quad
 (\mathcal{P}_{\tilde{V}})^2=\sum_{i,j,a} e^a_{ij}\otimes \big(\sum_{k} \mathcal{A}_{\mathcal{P}_{\tilde{V}}} (e^a_{ki})* \mathcal{A}_{\mathcal{P}_{\tilde{V}}}(e^a_{jk})\big)
\end{align*}
where $*$ is the multiplication in $M^{op}$. As $\mathcal{P}_{\tilde{V}}=(\mathcal{P}_{\tilde{V}})^2$, it further follows that,
 
\begin{align*}
\mathcal{A}_{\mathcal{P}_{\tilde{V}}}(e^a_{ji})=\sum_k \mathcal{A}_{\mathcal{P}_{\tilde{V}}}(e^a_{jk})\mathcal{A}_{\mathcal{P}_{\tilde{V}}} (e^a_{ki})=m\:(\mathcal{A}_{\mathcal{P}_{\tilde{V}}}\otimes \mathcal{A}_{\mathcal{P}_{\tilde{V}}})\: m^*(e^a_{ji}).
\end{align*}
As above identity holds for all $e^a_{ji}$'s, our claim follows.
\end{proof}
We conclude this section with descriptions of these adjacency operators for a classical labeled multigraph.

\textbf{The classical case:}
For a classical multi-relation or a labeled multigraph $R\subseteq X\times X\times Y$, the multi-edge indicator is given by, 
\begin{align*}
\mathcal{P}_R=\sum_{(x_1,x_2,y)\in R} e_{x_1 x_1}\otimes e_{x_2x_2}\otimes e_{yy}\otimes e_{yy}\in l^{\infty}(X)\otimes l^{\infty}(X)\otimes l^{\infty}(Y)\otimes l^{\infty}(Y).
\end{align*}
The multi-adjacency operator is an element of $B(L^{2}(X)\otimes L^{2}(Y))$ and the adjacency operators of the underlying weighted and non weighted single-edged graphs lie in $B(l^{\infty}(X)).$ Their matrix coefficients are given by:
\begin{align*}
    (\mathcal{A}_{\mathcal{P}_R})_{(x_1x_2)(y_1y_2)}&=\delta_{y_1y_2} 1 \quad \text{if}\quad (x_1,x_2,y)\in R\\ 
    &=0\quad\text{otherwise.}\\
    (\mathcal{A}_{\mathcal{S}_R})_{x_1x_2}&=\text{number of edges from $x_1$ to $x_2$}\\
    (\mathcal{A}_{\mathcal{P}_{\tilde{R}}})_{x_1x_2}&=1\quad\text{if there exists $y\in Y$ such that}\quad (x_1,x_2,y)\in R;\\
    &=0 \quad\text{otherwise}.
\end{align*}

\section{\textbf{Decomposable Multi-relations}}\label{decomp_mult}
In this section, we will study a special class of multi-relations namely, ``decomposable multi-relations". We will see that any decomposable multi-relation, which is ``symmetric",  is confusability multigraph of some completely positive map (Theorem \ref{fin_thm}).\par 
Instead of working with two arbitrary pair of algebras $(M,N)$ where $M\subseteq B(H)$ and $N\subseteq B(K)$, for now, We assume, that 
$N$ is a full matrix algebra rather than a direct sum of matrix algebras; that is, 
$N=B(K)$ for some Hilbert space $K$.

\begin{nota}\label{nota_swap}
    Let $H$ and $K$ be two finite dimensional Hilbert spaces. Equipping $B(H)$ and $B(K)$ with Hilbert Schmidt inner product, We have following Hilbert space isomorphisms:
    \begin{align*}
B(H\otimes K)\cong H\otimes \overline{H}\otimes K\otimes \overline{K}\cong H\otimes K\otimes \overline{K}\otimes \overline{H} \cong B(\overline{K},H)\otimes B(H,\overline{K})
    \end{align*}
    
    We denote this chain of isomorphisms simply by the flip $\sigma:B(\overline{K},H)\otimes B(H,\overline{K})\rightarrow B(H\otimes K)$.
\end{nota}
\begin{defn}\label{decomposable_multi}
    A quantum multi-relation $V$ on $(M,N)$ is said to be \textbf{decomposable} if there exists $V_1\subseteq B(\overline{K},H)$ and $V_2\subseteq B(H,\overline{K})$ such that, 
    \begin{align*}
        \sigma(V_1\otimes V_2)=V
    \end{align*}
    where the  $\sigma$ is defined in notation \ref{nota_swap}.
\end{defn}
\begin{rem}
     Definition \ref{decomposable_multi} does not behave well with respect to amplifications of the algebra $N$. For example, let us consider $M=l^\infty(X) $ and $N=l^{\infty}(Y)$ where $X=\{1,2\}$ and $Y=\{1\}$. The trivial multi-relation $V=B(l^2(X))\otimes l^{\infty}(Y)$ is decomposable if we consider $l^{\infty}(Y)\subseteq B(l^2(Y))\cong B(\mathbb{C})$. The moment we amplify $l^{\infty}(Y)$, by embedding $l^{\infty}(Y)$ in $B(\mathbb{C}^2)$ as $\{\lambda \:id_{\mathbb{C}^2}\:|\:\lambda\in\mathbb{C}\}$, the amplification of $V$ is not decomposable anymore. However, decomposability behaves well with respect to amplifications of $M$. The proof is similar to the proof of theorem \ref{embedding_indep}.
\end{rem}
For a decomposable multi-relation $V$, its components $V_1$ and $V_2$ are $M'$ left and right modules respectively. On the other hand, for any algebra $N\subseteq B(K)$, using remark \ref{opp_action}, we observe that  $N$ has a right action on $\overline{K}$ given by $\overline{\xi}*T=\pi_{op}(T)(\xi)=\overline{T^*\xi}$. Let us consider the the Weaver action of $M\otimes N$ on $B(\overline{K},H)$ by the following description:
\begin{align*}
    \pi_1(T_1\otimes T_2)(S)=T_1S\: \pi_{op}(T_2)\quad \text{where}\quad T_1\otimes T_2\in M\otimes N, S\in B(\overline{K},H)
\end{align*}
Let $\mathcal{P}_{V_1}=1-\tilde{\mathcal{P}_{V_1}}$ where $\tilde{\mathcal{P}_{V_1}}$ is the generator of the annihilator ideal of $V_1$ under the Weaver action. Using lemma \ref{left_action_annhi_proj}, it follows that $\mathcal{P}_{V_1}$ is an orthogonal projection from $B(\overline{K},H)$ onto $V_1$.  Similarly, by considering the action of $N^{op}\otimes M^{op}$ on $B(H,\overline{K})$ we get the orthogonal projection $\mathcal{P}_{V_2}$ from $B(H,\overline{K})$ onto the subspace $V_2$.
\begin{prop}\label{decom_mult_idn}
Let $V\subseteq B(H)\otimes B(K)$ be a decomposable multi-relation and with components $V_1\subseteq B(\overline{K},H)$ and $V_2\subseteq B(H,\overline{K})$. Then following conditions hold:
\begin{enumerate}
    \item $\sigma'(\mathcal{P}_{V_1}\otimes \mathcal{P}_{V_2})=\mathcal{P}_V$ where $\sigma'$ is an isomorphism between $M\otimes N\otimes N^{op}\otimes M^{op}$ and $M\otimes M^{op}\otimes N\otimes N^{op}$ obtained by swapping tensor components.
   
    \item If $M$ is minimally represented on $H$, then  $\mathcal{A}_{\mathcal{S}_{V}}=\mathcal{A}_{\mathcal{P}_{V_2}}\circ \mathcal{A}_{\mathcal{P}_{V_1}}$ where $\mathcal{S}_{V}$ is the weighted edge indicator (remark \ref{weighted_edge_indicator_rem}) of the multigraph $V$  .
\end{enumerate}
\end{prop}
\begin{proof}
Let us assume $\mathcal{P}_{V_1}=\sum_{i} P_i\otimes Q_i$ and $\mathcal{P}_{V_2}=\sum_{j} P'_j\otimes Q'_j$. Then it follows that 
\begin{align}\label{projector_tensor_swap}
 \sigma'(\mathcal{P}_{V_1}\otimes \mathcal{P}_{V_2})=\sum_{i,j} P_i\otimes Q'_j\otimes Q_i \otimes P'_j .   
\end{align}
We denote $\pi_1$, $\pi_2$ and $\pi$ to be the Weaver actions of $M\otimes N$, $N^{op}\otimes M^{op}$ and $M\otimes M^{op}\otimes N\otimes N^{op}$ on $B(\overline{K},H)$, $B(H,\overline{K})$ and $B(H)\otimes B(K)$ respectively.  For $\xi_1,\xi_2\in H$ and $\eta_1, \eta_2\in K$,  we observe that, $$\sigma(\theta_{\xi_1, \overline{\eta_1}}\otimes \theta_{\overline{\eta_2},\xi_2})=\theta_{\xi_1\xi_2}\otimes \theta_{\eta_1,\eta_2}\in B(H)\otimes B(K).$$
It follows that,
\begin{align}
    \sigma\circ (\pi_1(\mathcal{P}_{V_1})\otimes \pi_2(\mathcal{P}_{V_2}))(\theta_{\xi_1, \overline{\eta_1}}\otimes \theta_{\overline{\eta_2},\xi_2})&=\sigma(\sum_{i,j} \theta_{(P_i \xi_1), (\overline{Q_i \eta_1})}\otimes \theta_{(\overline{P'^*_j \eta_2}), (Q'^*_j \xi_2)})\notag\\
    &=\sum_{i,j}\theta_{(P_i \xi_1),(Q'^*_j \xi_2)}\otimes \theta_{(Q_i \eta_1), (P'^*_j \eta_2)}\notag\\
    &=\sum_{i,j} P_i \theta_{\xi_1,\xi_2} Q'_j \otimes Q_i \theta_{\eta_1,\eta_2} P'_j\notag\\
    &=\pi(\sigma'(\mathcal{P}_{V_1}\otimes \mathcal{P}_{V_2}))(\theta_{\xi_1,\xi_2}\otimes \theta_{\eta_1\eta_2}). \label{swap_eqiv_actions}
\end{align}
As $\pi_1(\mathcal{P}_{V_1})\otimes \pi_2(\mathcal{P}_{V_2})$ projects $B(\overline{K}, H)\otimes B(H,\overline{K})$ onto $V_1\otimes V_2$, we observe that, $\sigma\circ (\pi_1(\mathcal{P}_{V_1})\otimes \pi_2(\mathcal{P}_{V_2}))\circ \sigma^{-1}$ orthogonally projects $B(H)\otimes B(K)$ onto $V$. As $\pi(\mathcal{P}_V)$ also projects $B(H)\otimes B(K)$ onto $V$, from equation \ref{swap_eqiv_actions}, (1) follows. \par 
 The adjacency operators $\mathcal{A}_{\mathcal{P}_{V_1}}: M\rightarrow N^{op} $ and $\mathcal{A}_{\mathcal{P}_{V_2}}:N^{op}\rightarrow M$ are given by,
\begin{align*}
    \mathcal{A}_{\mathcal{P}_{V_1}}(m)=\sum_i tr_H(P_im)Q_i\quad\text{and}\quad \mathcal{A}_{\mathcal{P}_{V_2}}(n)&=\sum_j tr_{\overline{K}}(\pi_{op}(P'_j*n))Q'_j\\
    &=\sum_j tr_K (nP'_j)Q'_j
\end{align*}
To prove (2),  using proposition \ref{P_v-S_v_relation} and equation \ref{projector_tensor_swap} we observe that, for $m\in M$,
\begin{align*}
    \mathcal{A}_{\mathcal{P}_{V_2}}\circ  \mathcal{A}_{\mathcal{P}_{V_1}}(m)= \sum_i\mathcal{A}_{\mathcal{P}_{V_2}}(tr_H(P_im)Q_i)=\sum_{i,j} tr_H(P_i m)tr_K(Q_i P'_j)Q'_j=\mathcal{A}_{\mathcal{S_V}}(m).
    \end{align*}
 \end{proof}
 \begin{rem}
While preparing this paper, we came across a recent work \cite{daws2026quantum} in which the author introduces the notion of composition of two bipartite or generalized quantum relations. In our setting, the components of a decomposable multi-relation can be interpreted as two bipartite quantum relations. Moreover, part (2) of Proposition \ref{decom_mult_idn} shows that the weighted adjacency operator of the multigraph coincides with the adjacency operator arising from the composition of these quantum relations.
 \end{rem}
\begin{prop}\label{symm_decomp_char}
  Let $V$ be a decomposable quantum multi-relation. Then following conditions are equivalent:
  \begin{enumerate}
      \item $V$ is a symmetric  quantum multi-relation. 
      \item $V_1=V^*_2$ where $V_1\subseteq B(\overline{K},H)$ and $V_2\subseteq B(H,\overline{K})$ are components of $V$. 
\end{enumerate}
Moreover, if $V$ is symmetric, then $V^2\subseteq V$, that is, $V$ is also transitive.
\end{prop}

\begin{proof}
We observe that the diagram in figure \ref{Components of $V^*$} commutes 
\begin{figure}[h]
    \centering
  \begin{tikzcd}
T_1 \otimes T_2
\arrow[r, "\sigma"]
\arrow[d]
&
\sigma(T_1 \otimes T_2)
\arrow[d, "(\cdot)^*"]
\\
T_2^* \otimes T_1^*
\arrow[r, "\sigma"]
&
(\sigma(T_1 \otimes T_2))^*
\end{tikzcd}
    \caption{Components of $V^*$ for a decomposable multi-relation $V$}
    \label{Components of $V^*$}
\end{figure}
where $T_1\in B(\overline{K},H)$ and $T_2\in B(H,\overline{K})$ and $\sigma:B(\overline{K},H)\otimes B(H,\overline{K})\rightarrow B(H)\otimes B(K)$ defined in notation \ref{nota_swap}. Therefore \begin{align}\label{V^*_decomp}
  V^*=\sigma(V^*_2\otimes V^*_1)  
\end{align}. Let us assume that $V$ is symmetric, that is, $V=V^*$.\\
\textbf{Claim:} $V_1=V^*_2$\\
\textbf{Proof of claim}: We define $W=V_1\otimes V_2=V^*_2\otimes V^*_1$. For a nonzero $S\in V_2$, let us define $\tau_S:W\rightarrow B(\overline{K},H)$ by $\tau_S(T_1\otimes T_2)=T_1\:tr_H(S^* T_2)$. As $S\in V_2$, $\tau_S(W)=V_1$. On the other hand, considering $W=V^*_2\otimes V^*_1$, $\tau_S(W)\subseteq V^*_2$. Therefore $V_1\subseteq V^*_2$. Conversely, by taking $S\in V^*_1$, we conclude that $V^*_2\subseteq V_1$. \par 
Conversely, if $V_1=V^*_2$, using equation \ref{V^*_decomp} it follows that $V$ is symmetric. \par
Now we prove that any symmetric decomposable quantum multi-relation is transitive. Using the previous part, it follows that, $V=\sigma(V^*_2\otimes V_2)$. Let us write
$V_2=span\:\{F_k\:|\:k\:\}$ where $\{F_k\:|\:k\}$ is a basis of $V_2$. Therefore,
\begin{align}\label{k,l_span_V*times_V}
V^*_2\otimes V_2=span\{F^*_k\otimes F_l\:|\:k,l\}
\end{align}

We recall that 
\begin{align}
   H\otimes K\ni vec(F^*_k)&=\sum_{i,j} (F^*_k)_{ij} e_i\otimes f_j=\sum_i e_i\otimes \overline{F_k(e_i)};\notag\\
   \overline{K}\otimes \overline{H}\ni vec(F_l)&=\sum_{j,i}(F_l)_{ji} \overline{f_j}\otimes \overline{e_i}=\sum_{j} \overline{f_j}\otimes \overline{F^*_l (\overline{f_j})}.\notag
\end{align} where $e_i$'s are orthonormal basis of $H$ and $f_j$'s are orthonormal basis of $K$. Using above expressions it follows that,
\begin{align*}
\sigma (F^*_k\otimes F_l)&=\sigma \big( \sum_{i,j} \theta_{e_i, (F_ke_i)} \otimes \theta_{\overline{f_j}, (F^*_l(\overline{f_j}))}\big)\\
&=\sum_{i,j} \theta _{e_i, (F^*_l(\overline{f_j}))} \otimes \theta_{\overline{F_ke_i}, f_j}\in B(H)\otimes B(K)\\
&=\sum_{i,j}\theta_{(e_i\otimes \overline{F_ke_i}), (F^*_l(\overline{f_j})\otimes f_j)}=\theta_{vec(F^*_k), vec(F^*_l)}\in B(H\otimes K).
\end{align*}
As $V=\sigma(V^*_2\otimes V_2)$, from \ref{k,l_span_V*times_V} and above computations it follows that,
\begin{align}\label{sym_decmp_form}
    V=span\:\{\theta_{vec(F^*_k), vec(F^*_l)}\:|\:k,l\}
\end{align}
Now it is easy to see that $V^2\subseteq V$.
\end{proof}
The notion of decomposability can be generalized when $N$ is direct sum of matrix algebras. Let $N=\bigoplus B(K_b)$ where $K_b$'s are finite dimensional Hilbert spaces and $K=\bigoplus_b K_b$. A multi-relation $V$ on $(M,N)$ is said to be ``decomposable" if $V_b:=(id\otimes j_b)V$ is decomposable for each $b$ where $j_b:K\rightarrow K_b$ is the orthogonal projection.  
We conclude this section with the following theorem.
\begin{thm} \label{fin_thm}
Let $M=\bigoplus_a B(H_a)$ and $N=\bigoplus_bB(K_b)$ be two finite dimensional algebras and $V\subseteq B(H)\otimes N^{op}$ be a symmetric decomposable quantum multi-relation where $H=\bigoplus_a H_a$. Then there exists a CP map $\Phi:M\rightarrow N$ such that $\tilde{S}_{\Phi}=V$ where $\tilde{S_{\Phi}}$ is the confusability multigraph of $\Phi$. 
\end{thm}
\begin{proof}
We observe that $N^{op}\cong \bigoplus_b B(\overline{K_b})$ via the isomorphism $\pi_{op}$ (remark \ref{opp_action}). As $V$ is  decomposable and symmetric, that is, each $V_b$ is decomposable and symmetric, using proposition \ref{symm_decomp_char}, it follows that there exists $W_b\subseteq B(H, K_b)$ such that, $$\sigma(W_b^*\otimes W_b)=V_b.$$ Let $\{F_{bk}\: |\: k\}$ be a spanning set for $W_b$. As $W_b$ is also a right $M'$ module, we can assume tha $F_{bk}$'s have the ``block form", that is, for each $F_{bk}$, there exists $a$ such that $Ker (F_{bk})\supseteq H^{\perp}_a$. Using equation \ref{sym_decmp_form} we observe that,
\begin{align}
    V=\bigoplus_bV_b&=\bigoplus_bspan\:\{\theta_{vec(F^*_{bk}), vec(F^*_{bl})}\:|\:k,l\}\notag\\
    &=\bigoplus_b span\:\big\{\sum_{i,j,a,a'}e^{aa'}_{ij}\otimes \theta_{({\overline{F_{bk}e^a_i})}, ({\overline{F_{bl}e^{a'}_j})}}\:|\:k,l\big\}.\label{final_thm_eq}
\end{align}
We define a completely positive map $\Phi:M\rightarrow N$ by
\begin{align*}
    \Phi(x)=\sum_{b,k} j_bF_{bk}\: x\: F^*_{bk}j^*_b\quad\text{where}\quad x\in M.
\end{align*}
Using theorem \ref{multigraph_kraus_form}, it follows that,
\begin{align*}
    \tilde{S}_{\Phi}= span \:\{\sum_{i,a,j,a'} {e^{aa'}_{ij}}\otimes \theta_{(F_{bl}e^{a'}_j), (F_{bk}e^a_i)}\:|\:b,k,l\}
\end{align*}
As we have identified $N^{op}$ with $\bigoplus_b B(\overline{K_b})$ via the identification $\theta_{\xi,\eta}\mapsto \theta_{\overline{\eta}, \overline{\xi}}$, using equation \ref{final_thm_eq}, the theorem follows.
\end{proof}

\bibliographystyle{alphaurl}
\bibliography{my_bib}

\end{document}